\documentclass[10pt]{amsart}
\usepackage{amsxtra, amsfonts, amsmath, amsthm, amstext, amssymb, amscd, mathrsfs, verbatim, color}
\usepackage{threeparttable}
\usepackage{mathtools}
\usepackage[ansinew]{inputenc}\usepackage[T1]{fontenc}
\usepackage[all,cmtip]{xy}
\theoremstyle{plain}

\newtheorem{theorem}{Theorem}[section]
\newtheorem{lemma}[theorem]{Lemma}
\newtheorem{definition-theorem}[theorem]{Definition-Theorem}
\newtheorem{proposition}[theorem]{Proposition}
\newtheorem{corollary}[theorem]{Corollary}

\theoremstyle{definition}
\newtheorem{definition}[theorem]{Definition}
\newtheorem{example}[theorem]{Example}
\newtheorem{remark}[theorem]{Remark}
\newtheorem{notation}[theorem]{Notation}
\newcommand \bth[1] { \begin{theorem}\label{t#1} }
\newcommand \ble[1] { \begin{lemma}\label{l#1} }

\newcommand \bpr[1] { \begin{proposition}\label{p#1} }
\newcommand \bco[1] { \begin{corollary}\label{c#1} }
\newcommand \bde[1] { \begin{definition}\label{d#1}\rm }
\newcommand \bex[1] { \begin{example}\label{e#1}\rm }
\newcommand \bre[1] { \begin{remark}\label{r#1}\rm }

\newcommand \bnota[1] {\begin{notation}\label{n#1}\rm }
\newcommand {\ele} { \end{lemma} }

\newcommand {\epr} { \end{proposition} }
\newcommand {\eco} { \end{corollary} }
\newcommand {\ede} { \end{definition} }
\newcommand {\eex} { \end{example} }
\newcommand {\ere} { \end{remark} }
\newcommand {\enota} { \end{notation} }

\begin{document}
\title[Product Functor]{On the Product functor on inner forms of the general linear group over a non-Archimedean local field} 

\author[Kei Yuen Chan]{Kei Yuen Chan}
\address{Department of Mathematics, The University of Hong Kong \\ Shanghai Center for Mathematical Sciences, Fudan University }
\email{kychan1@hku.hk}
\maketitle

\begin{abstract}
Let $G_n$ be an inner form of a general linear group over a non-Archimedean local field. We fix an arbitrary irreducible representation $\sigma$ of $G_n$. Building on the work of Lapid-M\'inguez on the irreducibility of parabolic inductions, we show how to define a full subcategory of the category of smooth representations of some $G_m$, on which the parabolic induction functor $\tau \mapsto \tau \times \sigma$ is fully-faithful. A key ingredient of our proof for the fully-faithfulness is constructions of indecomposable representations of length 2.

Such result for a special situation has been previously applied in proving the local non-tempered Gan-Gross-Prasad conjecture for non-Archimedean general linear groups. In this article, we apply the fully-faithful result to prove a certain big derivative arising from Jacquet functor satisfies the property that its socle is irreducible and has multiplicity one in the Jordan-H\"older sequence of the big derivative.
\end{abstract}

\section{Introduction}

Let $F$ be a non-Archimedean local field and let $D$ be a finite-dimensional $F$-central division algebra. Let $G_n=\mathrm{GL}_{n}(D)$ be the general linear group over $D$. Let $\mathrm{Alg}(G_n)$ be the category of smooth representations of $G_n$ over $\mathbb C$. The parabolic induction is an important tool in constructing representations and plays a central role in the Zelevinsky classification of irreducible representations of $\mathrm{GL}_{n}(F)$ \cite{Ze80}. Recently, Aizenbud-Lapid and Lapid-M\'inguez \cite{LM16, LM19, LM20, AL22, LM22} extensively study the irreducibility of parabolic inductions, with rich connections to combinatorics and geometry.



This paper focuses on some homological aspects of parabolic inductions. The main purpose is to elaborate some observations and results in \cite{Ch22}, which we use functorial properties of parabolic inductions for proving the local non-tempered Gan-Gross-Prasad conjecture \cite{GGP20}. Our main result addresses the remark in \cite[Section 9.2]{Ch22}.

We first explain the main object-- the product functor. We denote by $\times$ the normalized parabolic induction. For a fixed irreducible representation $\pi$ and a full subcategory $\mathcal A$ of $\mathrm{Alg}(G_n)$, define
\[  \times_{\pi, \mathcal A}: \mathcal A \rightarrow \mathrm{Alg}(G_{n+k}), 
\]
given by $\times_{\pi, \mathcal A}(\omega) =\pi \times \omega$. Here we regard $\times_{\pi, \mathcal A}$ as a functor such that for a morphism $f: \pi' \rightarrow \pi''$ in $\mathcal A$, $\times_{\pi, \mathcal A}(f)=\mathrm{Id}_{\pi} \times f$, the one induced from the parabolic induction (see Section \ref{ss product functor} for more precise descriptions).

 Some general results about the product functor with respect to smooth duals and cohomological duals are given in Section \ref{s general product functor}.

We briefly recall the Zelevinsky theory \cite{Ze80} for $D=F$, see Section \ref{ss basic notations} for more notations. A {\it segment} takes the form $[a,b]_{\rho}$ for a supercuspidal representation $\rho$ of some $G_m$ and $a,b \in \mathbb C$ with $b-a\in \mathbb Z_{\geq 0}$. Zelevinsky \cite{Ze80} associates each segment $\Delta$ with a representation $\langle \Delta \rangle$, called a {\it segment representation}. A {\it multisegment} is a multiset of segments. Let $\mathrm{Mult}$ be the set of multisegments. For $\mathfrak m \in \mathrm{Mult}$, let $\langle \mathfrak m \rangle$ be the associated Zelevinsky module \cite{Ze80}.

The irreducibility of the parabolic induction is extensively studied in \cite{LM16}. A first question is that for a given irreducible representation $\pi$ of $G_n$, how one can find another irreducible representation $\pi'$ of $G_m$ such that $\pi \times \pi'$ is also irreducible. One way to do so is via 'building from the (basic) segment case'. The precise meaning is as follows. Set
\[  \mathcal M_{\pi}=\left\{ \mathfrak n \in \mathrm{Mult}: \langle \Delta \rangle \times \pi \mbox{ is irreducible }  \forall \Delta \in \mathfrak n \right\} .
\]
Then, for any $\mathfrak n \in \mathcal M_{\pi}$,  $\langle \mathfrak n \rangle \times \pi$ is irreducible \cite[Proposition 6.1]{LM16}. The converse is not true in general i.e. if $\langle \mathfrak n \rangle \times \pi$ is irreducible, it is not necessary that $\mathfrak n \in \mathcal M_{\pi}$. 

We write $\mathfrak m_1 \leq_Z \mathfrak m_2$ if $\mathfrak m_1$ is obtained from $\mathfrak m_2$ by a sequence of intersection-union operations (see Section \ref{ss intersection union}). Our observation is that the set $\mathcal M_{\pi}$ is closed under intersection-union operations in the following sense:
\begin{theorem} (=Theorem \ref{thm intersection-union closed}) \label{thm closed zel segment case}
Let $\pi$ be an irreducible representation of $G_n$. For $\mathfrak n \in \mathcal M_{\pi}$, if $\mathfrak n'$ is another multisegment with $\mathfrak n' \leq_Z \mathfrak n$, then $\mathfrak n' \in \mathcal M_{\pi}$. 
\end{theorem}
Our proof for Theorem \ref{thm closed zel segment case} uses properties from intertwining operators on $\square$-irreducible representations.  Another possible approach for proving Theorem \ref{thm closed zel segment case} is to use the combinatorial criteria of Lapid-M\'inguez in \cite[Proposition 5.12]{LM16}.

We now set $\mathcal A_{\pi}=\mathrm{Alg}_{\pi}(G_n)$ to be the full subcategory of $\mathrm{Alg}(G_n)$ whose objects are of finite length and have all simple composition factors isomorphic to $\langle \mathfrak m\rangle$ for some $\mathfrak m \in \mathcal M_{\pi}$. The significance of Theorem \ref{thm closed zel segment case} is that one can obtain plenty examples of extensions from the set $\mathcal M_{\pi}$ and so $\mathcal A_{\pi}$ is not semisimple in most of cases. Indeed, those extensions are preserved under $\times_{\pi, \mathcal A_{\pi}}$, shown in Proposition \ref{prop indecomp non-isomorphic} and Theorem \ref{thm indecomposable length 2}. This in turn implies our main result:
\begin{theorem} \label{thm fully fiathful functor} (=Theorem \ref{thm fully faithful functor})
Let $\pi$ be an irreducible representation of $G_n$. Then $\times_{\pi, \mathcal A_{\pi}}$ is a fully-faithful functor.
\end{theorem}
 \cite{Ch22} deals with the case that $\pi$ is a Speh representation and $\mathcal A$ is some subcategory coming from the irreducibility of the product between a cuspidal representation and $\pi$. 

A key new ingredient in the proof of Theorem \ref{thm fully fiathful functor} is a construction of extensions between two irreducible representations. This differs from the approach used in \cite{Ch22}, although we also need a basic case (when $\pi$ is also a segment representation) from \cite{Ch22}. The main idea comes from a study of first extensions in the graded Hecke algebra case in \cite{Ch18}. Roughly speaking, those extensions for two non-isomorphic representations come from Zelevinsky standard modules, and those for two isomorphic representations reduce to the tempered case. However, we remark that we do not have a concrete classification for indecomposable modules of length $2$.

For the self-extension case, we actually have more general statement:

\begin{theorem} (=Theorem \ref{thm indecomposable length 2}) \label{thm preserve self extensions}
Let $\pi_1$ and $\pi_2$ be irreducible representations of $G_k$ and $G_l$ respectively such that $\pi_1 \times \pi_2$ is still irreducible. Suppose $\lambda$ is an indecomposable representation of length $2$ with both simple composition factors isomorphic to $\pi_2$. Then $\pi_1 \times \lambda$ is also indecomposable.
\end{theorem}
Perhaps an interesting point of Theorem \ref{thm preserve self extensions} is that the parabolic induction does not preserve indecomposability in general. In other words, some non-trivial self-extensions can be trivialized under parabolic inductions (see Remark \ref{rmk example no preserve self ext}).

Theorem \ref{thm preserve self extensions} concerns about indecomposable modules of length $2$. Our proof relies on some constructions of those modules. One important ingredient is analogous properties in the affine highest weight category introduced by Kleshchev \cite{Kl15} (also see \cite{Ka17}), see the proofs in Section \ref{ss indecomp noniso case}. Roughly speaking such ingredient reduces to the computations of Ext-groups for tempered modules. Such Ext-groups are now better understood due to the work on discrete series by Silberger, Meyer, Opdam-Solleveld \cite{Si79, Me06, OS09} using analytic methods and by \cite{Ch16} using algebraic methods; and more general case \cite{OS13} via $R$-groups. We also refer the reader to \cite{Ch18} for more discussions.

Recent articles \cite{LM16, LM19, LM20, AL22, LM22} study the conditions of irreducibility for more general multisegments. In particular, when one of the multisegments arises from a so-called $\square$-irreducible representation, there are some precise conjectures connecting to the geometry of nilpotent orbits due to Gei\ss-Leclerc-Schr\"oer and Lapid-M\'inguez \cite{GLS11, LM19, LM20}. Thus one may hope for a version of Theorem \ref{thm fully fiathful functor} for replacing the segment case by other interesting classes of representations such as Speh, ladder or even $\square$-irreducible representations. One main problem goes back to understand the analog of the set $\mathcal M_{\pi}$ in Theorem \ref{thm closed zel segment case} and so $\mathrm{Alg}_{\pi}(G_n)$ in Theorem \ref{thm fully fiathful functor}. In \cite{GL21}, Gurevich-Lapid introduce a new class of representations parabolically induced from ladder representations, and so it is natural to ask if the extensions arising from those standard representations can be used to define a suitable analogue of $\mathrm{Alg}_{\pi}(G_n)$. 



We now consider a Jacquet functor version of above discussions. For an irreducible representation $\pi$ of $G_k$ and for an admissible representation $\tau$ of $G_n$, define
\[   \mathbb D_{\pi}(\tau) := \mathrm{Hom}_{G_{k}}(\pi, \tau_{N^-}) ,
\]
where $N^-$ is the opposite unipotent radical of the standard parabolic subgroup in $G_n$ containing $G_{k}\times G_{n-k}$. Here $\tau_{N^-}$ is viewed as a $G_k$-module via the embedding $G_k \hookrightarrow G_{k}\times G_{n-k}$ to the first factor. We shall call such $\mathbb D_{\pi}$ to be a big derivative, and $\mathbb D_{\pi}(\tau)$ has a natural $G_{n-k}$-module structure (also see Definition \ref{def big derivatives}).

The big derivative $\mathbb D_{\pi}$ is the right adjoint functor of the product functor, if we consider the functors are for the category of all smooth representations. However, this is not entirely correct if we restrict the functor to the full subcategory $\mathrm{Alg}_{\pi}(G_n)$ defined above. Nevertheless, there are some interesting cases that $\mathbb D_{\pi}$ forms an adjoint functor for $\times_{\pi, \mathrm{Alg}_{\pi}(G_n)}$. For example, the case considered in \cite{Ch22} works. In those cases, we could deduce that the big derivative is irreducible (see Lemma \ref{lem big derivative st} for a precise statement). This is consequently applied to prove:

\begin{theorem} (c.f. Theorem \ref{thm generic socle irreducible}) \label{thm generic derivative SI property}
Suppose $D=F$. Let $\pi$ be a generic irreducible representation of $G_k$. Let $\tau$ be any irreducible representation of $G_n$ such that $\mathbb D_{\pi}(\tau)\neq 0$. Then $\mathbb D_{\pi}(\tau)$ satisfies the socle irreducible property (i.e. $\mathbb D_{\pi}(\tau)$ has unique simple submodule and such simple submodule appears with multiplicity one in the Jordan-H\"older sequence of $\mathbb D_{\pi}(\tau)$). 
\end{theorem}

When $\pi$ is a cuspidal representation, the analogous statement for Theorem \ref{thm generic derivative SI property} for the affine Hecke algebra of type A is shown in the work of Grojnowski-Vazirani \cite{GV01} by exploiting the explicit structure of a principal series. 

The irreducibility part of the socle in Theorem \ref{thm generic derivative SI property} is shown by Kang-Kashiwara-Kim-Oh \cite{KKKO15} (also see \cite{AL22}) in a greater generality on $\square$-irreducible representations. For the irreducibility part of the socle, some variants of more specific cases using Gelfand-Kazhdan method are also shown by Aizenbud-Gourevitch \cite{AG12}. 

Our emphasis on Theorem \ref{thm generic derivative SI property} is on the application of the product functor, which gives a basic case in Proposition \ref{prop multiplicity one segment case}. An advantage of this method is that one does not have to compute some internal structures of some modules. Hence, it has a higher potential for other applications such as the one in \cite{Ch22}. We shall also show how to extend the socle irreducible result to the case of generic representations (i.e. full version of Theorem \ref{thm generic derivative SI property}) in the appendix.

As an analog of the problem of studying the irreducibility of parabolic inductions, one may ask the irreducibility of big derivatives. The product functor provides a technique on such problem as shown in the article while Theorem \ref{thm generic derivative SI property} provides some concrete examples.


A more classical viewpoint on studying parabolic inductions and Jacquet functors is on the Grothendieck group of the category of smooth representations of $\mathrm{GL}_n(F)$'s, in which those functors give a Hopf algebra structure \cite{Ze80, Ze81, Ta90, Ta95, HM08}. We hope this work could emphasis on some interesting higher structures associated to parabolic inductions and Jacquet functors.

\subsection{Acknowledgment} 
The author would like to thank Erez Lapid for drawing attention to \cite[Section 7]{LM22}. The author would like to thank Max Gurevich for his inputs on Section \ref{s product with segment rep}  and discussions on \cite{GL21}. This project is supported in part by National Key Research and Development Program of China (Grant No. 2020YFA0713200),
the Research Grants Council of the Hong Kong Special Administrative Region, China (Project No: 	17305223) and NSFC grant for Excellent Young Scholar (Project No.: 12322120). The author would like to thank referees' comments on improving expositions of the article, and thank one of the referees for an input in the proof of Theorem \ref{thm intersection-union closed} that leads to a simplification. The author declares that he has no conflict of interest.

\section{Notations} \label{s notations}

\subsection{Basic notations} \label{ss basic notations}

Let $F$ be a non-Archimedean local field and let $D$ be a finite-dimensional $F$-central division algebra. Let $G_n=\mathrm{GL}_n(D)$, the general linear group over $D$. The group $G_0$ is viewed as the trivial group.  For $g \in G_n$, let $\nu(g)=|\mathrm{Nrd}(g)|_F$, where $\mathrm{Nrd}: G_n \rightarrow F^{\times}$ is the reduced norm and $|.|_F$ is the absolute value on $F$. All the representations we consider are smooth over $\mathbb C$ and we usually omit the adjectives 'smooth' and 'over $\mathbb C$'. For a representation $\pi$ of $G_n$, we write $\mathrm{deg}(\pi)$ for $n$. We shall usually not distinguish representations in the same isomorphism class. 

For a supercuspidal representation $\rho$ of $G_n$, let $s_{\rho}$ be the unique value in $\mathbb R_{>0}$ such that $\rho \times \nu^{\pm s_{\rho}}\rho$ is reducible. Set $\nu_{\rho}=\nu^{s_{\rho}}$. For $a, b\in \mathbb Z$ with $b-a\in \mathbb Z_{\geq 0}$ and a supercuspidal representation $\rho$, a {\it segment} $[a,b]_{\rho}$ is the set $\left\{ \nu_{\rho}^a\rho, \nu_{\rho}^{a+1}\rho, \ldots, \nu_{\rho}^b\rho \right\}$. We consider two segments $[a,b]_{\rho}$ and $[a',b']_{\rho'}$ are equal if $\nu_{\rho}^a\rho \cong \nu_{\rho'}^{a'}\rho'$ and $\nu_{\rho}^b\rho \cong \nu_{\rho}^{b'}\rho'$. If $\rho=1$ is the trivial representation of $G_1$, we may simply write $[a,b]$ for $[a,b]_1$. We also consider the empty set as a segment and also set $[a,a-1]_{\rho}=\emptyset$. The {\it absolute length} $l_{abs}([a,b]_{\rho})$ of a segment $[a,b]_{\rho}$ is $(b-a+1)\mathrm{deg}(\rho)$. A {\it multisegment} is a multiset of non-empty segments, and we also consider the empty set as a multisegment. For a multisegment $\mathfrak m=\left\{ \Delta_1, \ldots, \Delta_k \right\}$, define $l_{abs}(\mathfrak m)=l_{abs}(\Delta_1)+\ldots +l_{abs}(\Delta_k)$, called the 
{\it length} of $\mathfrak m$.

We introduce the following notations:
\begin{itemize}
\item $\mathrm{Irr}(G_n)=$ the set of irreducible smooth representations of $G_n$ and $\mathrm{Irr}=\cup_n \mathrm{Irr}(G_n)$; 
\item $\mathrm{Alg}(G_n)=$ the category of smooth representations of $G_n$;
\item $\mathrm{Alg}_f(G_n)=$ the full subcategory of $\mathrm{Alg}(G_n)$ of all the smooth $G_n$-representations of finite length;
\item let $\mathrm{Alg}_f$ be the set of smooth representations of some $G_n$ (in other words, it is the set of all objects in $\mathrm{Alg}_f(G_n)$ for some $n$);
\item $\mathrm{Seg}_n=$ the set of segments of absolute length $n$; and $\mathrm{Seg}=\cup_n \mathrm{Seg}_n$;
\item $\mathrm{Mult}_n=$ the set of multisegments of length $n$; and $\mathrm{Mult}=\cup_n \mathrm{Mult}_n$;
\item for $\pi \in \mathrm{Alg}_f$, $\mathrm{JH}(\pi)=$ the set of simple composition factors in $\pi$ (i.e. multiplicities are not counted);
\item ${}^{\vee}: \mathrm{Alg}(G_n) \rightarrow \mathrm{Alg}(G_n)$ is the smooth dual contravariant functor;
\item for $\Delta=[a,b]_{\rho} \in \mathrm{Seg}$, define 
\[ \Delta^{\vee}=[-b,-a]_{\rho^{\vee}} ; \]
 for $\mathfrak m=\left\{ \Delta_1, \ldots, \Delta_r \right\} \in \mathrm{Mult}$, 
\[ \mathfrak m^{\vee}=\left\{ \Delta_1^{\vee}, \ldots, \Delta_r^{\vee} \right\} \]
\item for two supercuspidal representations $\rho_1, \rho_2$, we write  $\rho_1 < \rho_2$ if there exists a positive integer $c$ such that $\rho_2 \cong \nu_{\rho_1}^c \rho_1$. We write $\rho_1 \leq \rho_2$ if either $\rho_1<\rho_2$ or $\rho_1\cong \rho_2$;
\item for a segment $\Delta=[a,b]_{\rho}$, write $a(\Delta)=\nu_{\rho}^a\rho$ and $b(\Delta)=\nu_{\rho}^b\rho$; 
\item for $\pi \in \mathrm{Irr}$, we write $\mathrm{csupp}(\pi)=\left\{ \sigma_1, \ldots, \sigma_r \right\}$ to be the unique multiset of supercuspidal representations such that $\pi$ is a composition factor of $\sigma_1 \times \ldots \times \sigma_r$;
\item for $\pi \in \mathrm{Alg}_f$, let $\mathrm{soc}(\pi)$ be the socle (i.e. maximal semisimple submodule) of $\pi$ and let $\mathrm{cosoc}(\pi)$ be the cosocle (i.e. maximal semisimple quotient) of $\pi$. 
\end{itemize}

Since we are working on representations over $\mathbb C$, we shall not distinguish cuspidal representations and supercuspidal representations.

\subsection{More notations for segments and multisegments} \label{ss intersection union}

For $\mathfrak m \in \mathrm{Mult}$, two segments $\Delta_1$ and $\Delta_2$ in $\mathfrak m$ are said to be {\it linked} if $\Delta_1 \cup \Delta_2$ is still a segment and $\Delta_1 \not\subset \Delta_2$ and $\Delta_2 \not\subset \Delta_1$. We write $\Delta_1 < \Delta_2$ if $\Delta_1$ and $\Delta_2$ are linked and $a(\Delta_1) < a(\Delta_2)$. Note $\Delta_1< \Delta_2$ automatically implies $b(\Delta_1)<b(\Delta_2)$.

A multisegment $\mathfrak m$ is said to be {\it generic} if any two segments in $\mathfrak m$ are not linked.

As in \cite{Ze80}, for $\mathfrak m, \mathfrak n \in \mathrm{Mult}_n$, we say that $\mathfrak m$ is obtained by an intersection-union process if there are two linked segments $\Delta_1, \Delta_2$ in $\mathfrak n$ such that 
\[  \mathfrak m=\left\{ \begin{array}{cc} \mathfrak n - \left\{ \Delta_1, \Delta_2 \right\}+ \left\{ \Delta_1 \cup \Delta_2, \Delta_1\cap \Delta_2 \right\} \quad & \mbox{ if $\Delta_1 \cap \Delta_2 \neq \emptyset$ } \\
                                      \mathfrak n - \left\{ \Delta_1, \Delta_2 \right\} +\left\{ \Delta_1\cup \Delta_2 \right\} \quad & \mbox{ if $\Delta_1 \cap \Delta_2 =\emptyset$ } 
																			 \end{array}  
																			\right.
\]
Here $+$ and $-$ represent the union and substraction as multisets.

Write $\mathfrak m <_Z \mathfrak n$ if $\mathfrak m$ can be obtained by a sequence of intersection-union operations from $\mathfrak n$, and write $\mathfrak m \leq_Z \mathfrak n$ if $\mathfrak m=\mathfrak n$ or $\mathfrak m <_Z \mathfrak n$. 

\subsection{Langlands and Zelevinsky classification}

For a segment $\Delta=[c,d]_{\rho}$, let $\langle \Delta \rangle$ (resp. $\mathrm{St}(\Delta)$) be the unique simple submodule (resp. quotient) of 
\[  \nu_{\rho}^c\rho \times \ldots \times \nu_{\rho}^d \rho .
\]

For any multisegment $\mathfrak m$, write the segments in $\mathfrak m$ as $\Delta_1, \ldots, \Delta_r$. We label the segments in $\mathfrak m$ such that 
\[  b(\Delta_1) \not< b(\Delta_2) \not< \ldots \not< b(\Delta_r). 
\]
Let
\[  \zeta(\mathfrak m) = \langle \Delta_1 \rangle \times \ldots \langle \Delta_r \rangle,\quad  \lambda(\mathfrak m)= \mathrm{St}(\Delta_1) \times \ldots \times \mathrm{St}(\Delta_r) .
\]
It is shown in \cite{Ze80} that there is a unique simple submodule in $\zeta(\mathfrak m)$, which will be denoted $\langle \mathfrak m\rangle$. It is independent of a choice of a labeling above. There is a one-to-one correspondence
\[  \mathrm{Mult}_n  \longleftrightarrow \mathrm{Irr}(G_n), \quad  \mathfrak m \mapsto \langle \mathfrak m \rangle  .
\]
The Zelevinsky classification is due to Zelevinsky \cite{Ze80} when $D=F$ and due to M\'inguez-S\'echerre \cite{MS13, MS14} when $D$ is general. 

On the other hand, $\lambda(\mathfrak m)$ has unique simple quotient, denoted by $\mathrm{St}(\mathfrak m)$. This gives another one-to-one correspondence 
\[  \mathrm{Mult}_n \longleftrightarrow \mathrm{Irr}(G_n), \quad \mathfrak m \mapsto\mathrm{St}(\mathfrak m) .
\]
The above correspondence in the form due to Langlands is known for $D=F$ in \cite{Ze80} and the general case follows from the local Jacquet-Langlands correspondence due to Deligne-Kazhdan-Vign\'eras \cite{BKV84} for the zero characteristic and Badulescu \cite{Ba02} for positive characteristics. Such classification also has significance in the unitary dual problem, see work of Tadi\'c, S\'echerre, Badulescu-Henniart-Lemaire-S\'echerre \cite{Ta90, Sc09, BHL10}. 

\subsection{Parabolic inductions and Jacquet functors} \label{ss parabolic induction jacquet}

For non-negative integers $n_1, \ldots, n_r$ with $n_1+\ldots+n_r=n$, let $P_{n_1, \ldots, n_r}$ be the parabolic subgroup containing the subgroup $G_{n_1}\times \ldots \times G_{n_r}$ as block diagonal matrices and all upper triangular matrices. We shall call $P_{n_1, \ldots, n_r}$ to be a standard parabolic subgroup. (Note that when $n_i$ is zero, $G_{n_i}$ is regarded as the trivial group and we may simply drop the term. We include such case for the convenience of notations later.) Let $N_{n_1, \ldots, n_r}$ be the unipotent radical of $P_{n_1, \ldots, n_r}$, and let $N_{n_1, \ldots, n_r}^-$ be the unipotent radical of the parabolic subgroup opposite to $P_{n_1, \ldots, n_r}$.

For $\pi_1 \in \mathrm{Alg}(G_{n_1})$ and $\pi_2 \in \mathrm{Alg}(G_{n_2})$, define $\pi_1 \times \pi_2$ to be 
\[  \pi_1 \times \pi_2 = \mathrm{Ind}_{P_{n_1, n_2}}^{G_{n_1+n_2}} (\pi_1\boxtimes \pi_2),
\]
the space of smooth functions from $G_{n_1+n_2}$ to $\pi_1 \boxtimes \pi_2$ satisfying 
\[f(pg)=\delta(p)^{1/2} p.f(g),\]
 where $\delta$ is the modular character of $P_{n_1, n_2}$. The $G_n$-action on $\pi_1 \times \pi_2$ is the right translation on those functions i.e. for $f \in \pi_1 \times \pi_2$,
\[   (g.f)(g')=f(g'g) . \]
 Here we consider $\pi_1\boxtimes \pi_2$ as a $P_{n_1, n_2}$-representation by the inflation. We shall simply call $\pi_1\times \pi_2$ to be a {\it product}. The product is indeed an associative operation and so there is no ambiguity in defining $\pi_1 \times \ldots \times \pi_r$.

For a parabolic subgroup $P$ of $G_n$ with the Levi decomposition $LN$, define the Jacquet functor, as a $L$-representation:
\[  \pi_{N} =\delta_{P}^{-1/2} \cdot \frac{\pi}{\langle n.x-x: x\in \pi, n \in N \rangle}, 
\]
where $\delta_P$ is the modular character of $P$.

Both parabolic inductions and Jacquet functors are exact functors. For $n_1+\ldots+n_r=n$, the parabolic induction $\pi \mapsto \mathrm{Ind}_{P_{n_1, \ldots , n_r}}^{G_n}\pi$ has the Jacquet functor $\pi \mapsto \pi_{N_{n_1, \ldots, n_r}}$ as its left adjoint functor, and has the opposite Jacquet functor $\pi \mapsto \pi_{N_{n_1, \ldots, n_r}^-}$ as its right adjoint functor.

Following \cite{LM19}, for $\pi \in \mathrm{Irr}$, $\pi$ is said to be {\it $\square$-irreducible} if $\pi \times \pi$ is irreducible. Let $\mathrm{Irr}^{\square}$ be the set of $\square$-irreducible representations. In the content of quantum affine algebras, it is called real modules, see e.g. work of Hernandez-Leclerc and Kang-Kashiwara-Kim-Oh \cite{HL13, KKKO15}. A particular class of $\square$-irreducible representations is those $\mathrm{St}(\mathfrak m)$ for a generic multisegment $\mathfrak m \in \mathrm{Mult}$.

\subsection{Geometric lemma} \label{ss geo lemma}

For $i\leq n$, we sometimes abbreviate $N_i$ for $N_{n-i, i} \subset G_n$. 

Let $\pi_1 \in \mathrm{Alg}(G_{n_1})$ and let $\pi_2 \in \mathrm{Alg}(G_{n_2})$. Let $n=n_1+n_2$. The geometric lemma \cite{BZ77} gives that $(\pi_1 \times \pi_2)_{N_i}$ admits a filtration, whose successive subquotients are of the form:
\[   \mathrm{Ind}_{P_{n_1-i_1,n_2-i_2}\times P_{i_1, i_2}}^{G_{n-i}\times G_i} ((\pi_1)_{N_{i_1}} \boxtimes (\pi_2)_{N_{i_2}})^{\phi} ,\]
where $i_1+i_2=i$ and $\phi$ sends a $G_{n_1-i_1}\times G_{i_1}\times G_{n_2-i_2}\times G_{i_2}$-representation to a $G_{n_1-i_1} \times G_{n_2-i_2}\times G_{i_1}\times G_{i_2}$-representation via the map
\[  (g_1, g_2, g_3, g_4) \mapsto (g_1, g_3, g_2, g_4) .
\]
Moreover, the bottom layer in the filtration of $(\pi_1 \times \pi_2)_{N_i}$ is when $i_1=\mathrm{min}\left\{n_1, i \right\}$ and the top layer in the filtration of $(\pi_1 \times \pi_2)_{N_i}$ is when $i_2=\mathrm{min}\left\{n_2, i \right\}$.

\subsection{Jacquet functors on segment and Steinberg representations} \label{ss jacquet segment steinberg}

Let $[a,b]_{\rho} \in \mathrm{Seg}$. Let $k=\mathrm{deg}(\rho)$. It follows from \cite[Propositions 3.4 and 9.5]{Ze80} and \cite[Proposition 3.1]{Ta90} that
\begin{align} \label{eqn jacquet on trivial repn}
      \langle [a,b]_{\rho} \rangle_{N_{ik}} \cong \langle [a,b-i]_{\rho} \rangle \boxtimes \langle [b-i+1, b]_{\rho} \rangle 
\end{align}
and
\begin{align} \label{eqn jacquet on steinberg repn}
     \mathrm{St}([a,b]_{\rho})_{N_{ik}} \cong  \mathrm{St}([a+i, b]_{\rho}) \boxtimes \mathrm{St}([a, a+i-1]_{\rho})  ,
\end{align}
and the Jacquet modules $\langle [a,b]_{\rho}\rangle_{N_j}$ and $\mathrm{St}( [a,b]_{\rho})_{N_j}$ are zero if $k$ does not divide $j$. 

We sometimes use the formulas implicitly in computing layers involving the geometric lemma.


\section{Some generalities of the product functor} \label{s general product functor}

\subsection{Product functor} \label{ss product functor}

Let $\mathcal A$ be a full Serre subcategory of $\mathrm{Alg}(G_n)$. For a fixed irreducible representation $\pi$ of $G_k$, we define the product functor 
\[ \times_{\pi, \mathcal A}: \mathcal A \rightarrow \mathrm{Alg}(G_{n+k})  \]
as
\[  \times_{\pi, \mathcal A}(\tau) =  \pi \times \tau  
\]
and, for a morphism $f$ from $\tau$ to $\tau'$ in $\mathcal A$ and $F \in \pi \times \tau$ (under the realization in Section \ref{ss parabolic induction jacquet}), 
\[ (\times_{\pi, \mathcal A}(f)(F))(g) = (\mathrm{Id}_{\pi}\boxtimes f)(F(g)) , \quad \mbox{for any $g \in G_{n+k}$ } .
\] 
Note that we do not assume $\times_{\pi, \mathcal A}$ preserves simple objects at this point.

\begin{proposition}
The functor $\times_{\pi, \mathcal A}$ is exact and faithful. 
\end{proposition}

\begin{proof}
Exactness follows from that the parabolic induction is an exact functor. The faithfulness then follows from that the functor sends a non-zero object to a non-zero object.
\end{proof}

\subsection{Smooth dual functor} \label{ss left right induction}

In this section, we specify to $D=F$. Let $\theta: G_n \rightarrow G_n$ given by $\theta(g)=g^{-t}$, the transpose inverse of $g$. This induces a covariant auto-equivalence for $\mathrm{Alg}(G_n)$, still denoted by $\theta$. For $D=F$, it is a classical result of Gelfand-Kazhdan that $\theta(\pi)\cong \pi^{\vee}$ for any $\pi \in \mathrm{Irr}$. 

\begin{definition}
A full Serre subcategory $\mathcal A$ of $\mathrm{Alg}_f(G_n)$ is said to be {\it $\sim\ $-closed} if for any object $C$ in $\mathcal A$, $\widetilde{C}$ is still in $\mathcal A$.
\end{definition}

One main example is the category $\mathrm{Alg}_{\pi}(G_n)$, as shown in Theorem \ref{thm intersection-union closed} later. Define $\widetilde{\ }= \theta \circ {}^{\vee}$ and so it is also a contravariant functor.

\begin{proposition} \label{prop fully faithful segment}
Let $D=F$. Let $\mathcal A$ be a $\sim$-closed full subcategory of $\mathrm{Alg}_f(G_n)$. Let $\pi \in \mathrm{Irr}(G_k)$. Define the right product functor 
\[ \times^{\pi, \mathcal A}: \mathcal A \rightarrow \mathrm{Alg}(G_{n+k}), \quad \times^{\pi, \mathcal A}(\pi')=\pi' \times \pi. \] 
Then $\times_{\pi, \mathcal A}$ is fully-faithful if and only if $\times^{\pi, \mathcal A}$ is fully-faithful. 
\end{proposition}

\begin{proof}
We only prove the if direction, and the only if direction can be proved similarly. We have the following isomorphisms:
\begin{align*}
   \mathrm{Hom}_{G_{n+k}}(\pi \times \pi_1', \pi \times \pi_2') 
\cong & \mathrm{Hom}_{G_{n+k}}(\widetilde{\pi \times \pi_2'}, \widetilde{\pi \times \pi_1'}) \\
\cong & \mathrm{Hom}_{G_{n+k}}(\widetilde{\pi}_2' \times \widetilde{\pi}, \widetilde{\pi}_1'\times \widetilde{\pi}) \\
\cong & \mathrm{Hom}_{G_{n+k}}(\widetilde{\pi}_2' \times \pi, \widetilde{\pi}_1' \times \pi) \\
\cong & \mathrm{Hom}_{G_n}(\widetilde{\pi}_2', \widetilde{\pi}_1') \\
\cong & \mathrm{Hom}_{G_n}(\pi_1', \pi_2')
\end{align*}
The first and last isomorphisms follow from taking the duals (and the representations are admissible). The second isomorphism follows from the compatibility between taking duals and parabolic inductions \cite[Page 173]{MW86}. The third isomorphism follows from a result of Gelfand-Kazhdan \cite[Theorem 2]{GK} (see \cite[Theorem 7.3]{BZ76}). The fourth isomorphism follows from the if direction.
\end{proof} 

\subsection{Cohomological dual functor}

Let $\mathcal H(G_n)$ be the space of compactly supported smooth $\mathbb C$-valued functions on $G_n$, viewed as a $G_n$-representation with the action given by: for $f \in \mathcal H(G_n)$, $(g.f)(g')=f(g'g)$. Let $\mathfrak R$ be a Bernstein component of $\mathrm{Alg}(G_n)$ and let $d$ be the homological dimension of $\mathfrak R$. Given a finitely-generated $G_n$-module $\pi$ in $\mathfrak R$, define
\[  \mathcal D(\pi) = \mathrm{Ext}_{G_n}^d(\pi, \mathcal H(G_n))
\]
viewed as a $G_n$-module. As shown in \cite[Page 102]{Be92} and \cite[Page 132]{SS97}, $\mathcal D$ is a contravariant exact functor. With the property that $\mathcal D^2=\mathrm{Id}$, $\mathcal D$ is a fully-faithful functor. The functor also sends a simple object to a simple object and agrees with the Aubert-Zelevinsky dual \cite{Ze80, Au95} in the Grothendieck group level, see the work of Schneider-Stuhler \cite[Proposition IV.5.2]{SS97} and Bernstein-Bezrukavnikov-Kazhdan \cite[Section 3.2]{BBK18}. Explicit algorithm for computing $\mathcal D(\pi)$ for $\pi \in \mathrm{Irr}$ is given by M\oe glin-Waldspurger \cite{MW86}.

We first recall the following result of Bernstein:

\begin{theorem}\cite[Theorem 31(4)]{Be92}  \label{thm fully faithful coh dual}
We fix Bernstein components $\mathfrak R_1$ and $\mathfrak R_2$ of $\mathrm{Alg}(G_{n_1})$ and $\mathrm{Alg}(G_{n_2})$ respectively. Let $\pi_1$ and $\pi_2$ be finitely-generated objects in $\mathfrak R_1$ and $\mathfrak R_2$ respectively. Let $\mathfrak R$ be the unique Bernstein component containing the object $\pi_1 \times \pi_2$. Then
\[  \mathcal D(\pi_1 \times \pi_2) \cong \mathcal D(\pi_2)\times \mathcal D(\pi_1). 
\]
\end{theorem}

We remark that the switch in the terms on the RHS comes from switching the induction between a standard parabolic subgroup and its opposite one.

\begin{corollary} \label{cor dual statement}
Let $\mathfrak R$ be a Bernstein component of $\mathrm{Alg}(G_n)$ and let $\mathfrak R_f$ be the full subcategory of $\mathfrak R$ of all objects of finite lengths. Let $\mathcal A$ be a full Serre subcategory of $\mathfrak R_f$. Let $\pi \in \mathrm{Irr}$. This gives a full subcategory $\mathcal D(\mathcal A)$ whose objects are $\mathcal D(\pi)$ for objects $\pi$ in $\mathcal A$ and morphisms $\mathcal D(f)$ for morphisms $f$ in $\mathcal A$. Then $\times_{\pi, \mathcal A}$ is a fully-faithful functor if and only if $\times^{\mathcal D(\pi), \mathcal D(\mathcal A)}$ is a fully-faithful functor. Here $\times^{\mathcal D(\pi), \mathcal D(\mathcal A)}$ is defined in Proposition \ref{prop fully faithful segment}.
\end{corollary}

\begin{proof}
This follows from Theorem \ref{thm fully faithful coh dual} and that $\mathcal D$ is a fully-faithful contravariant functor. 
\end{proof}

\section{Product with a segment representation and intersection-union process} \label{s product with segment rep}

Recall that for $\pi \in \mathrm{Irr}$, $\mathcal M_{\pi}$ is the set of all multisegments $\mathfrak n$ such that for any segment $\Delta$ in $\mathfrak n$, $\langle \Delta \rangle \times \pi$ is irreducible. 

We say that $\pi \in \mathrm{Alg}_f$ is {\it SI} or {\it socle irreducible} if $\mathrm{soc}(\pi)$ is irreducible and occurs with multiplicity one in the Jordan-H\"older sequence of $\pi$.

\begin{theorem} \label{thm intersection-union closed}
Let $\pi \in \mathrm{Irr}(G_n)$. Let $\mathfrak m \in \mathcal M_{\pi}$. For any $\mathfrak n \in \mathrm{Mult}$ with $\mathfrak n \leq_Z \mathfrak m$, $\mathfrak n \in \mathcal M_{\pi}$. 
\end{theorem}

\begin{proof}
For $\pi_1, \pi_2 \in \mathrm{Alg}_f$, let $R_{\pi_1, \pi_2}$ be the normalized non-zero intertwining operator from $\pi_1 \times \pi_2$ to $\pi_2 \times \pi_1$ (see \cite[Section 2]{LM19}). By the transitivity of $\leq_Z$, we reduce to the case that $\mathfrak m \in \mathcal M_{\pi}$ is of two linked segments. Now, fix an arbitrary $\pi \in \mathrm{Irr}$, and let $\Delta_1, \Delta_2 \in \mathrm{Seg}$ such that $\langle \Delta_1 \rangle \times \pi$ and $\langle \Delta_2 \rangle \times \pi$ are irreducible. Then $(R_{\langle \Delta_2\rangle \times \pi} \times \mathrm{Id}_{\langle \Delta_1\rangle})\circ (\mathrm{Id}_{\langle \Delta_2\rangle} \times R_{\langle \Delta_1 \rangle \times \pi})$ sends $\langle \Delta_2\rangle \times \langle \Delta_1 \rangle \times \pi$ to $\pi \times \langle \Delta_2 \rangle \times \langle \Delta_1 \rangle$, and is an isomorphism:
\begin{align} \label{eqn first appear in quotient}
   \langle \Delta_2 \rangle \times \langle \Delta_1\rangle \times \pi \cong \pi \times \langle \Delta_2 \rangle \times \langle \Delta_1 \rangle.
\end{align}
By switching the labelling if necessary, we also have:
\begin{align} \label{eqn first appear in submodule}
   \langle \Delta_1 \rangle \times \langle \Delta_2 \rangle \times \pi \cong \pi \times \langle \Delta_1 \rangle \times \langle \Delta_2 \rangle .
\end{align}
Again, by switching labelling if necessary, we may and shall assume that $\langle \Delta_1+\Delta_2 \rangle$ is in the quotient of $\langle \Delta_2 \rangle \times \langle \Delta_1 \rangle$.

Let $\tau=\langle \left\{ \Delta_1, \Delta_2 \right\}\rangle \times \pi$. Let 
\[  \tau_1:=\mathrm{soc}(\langle \Delta_1\cup \Delta_2+\Delta_1\cap \Delta_2 \rangle \times \pi), \quad \tau_2:= \mathrm{cosoc}(\langle \Delta_1\cup \Delta_2+\Delta_1 \cap \Delta_2 \rangle \times \pi) .
\]
Here $\Delta_1 \cup \Delta_2+\Delta_1\cap \Delta_2$ is equal to the multisegment $\left\{ \Delta_1\cup \Delta_2, \Delta_1 \cap \Delta_2 \right\}$ if $\Delta_1\cap \Delta_2$ is non-empty and is equal to the multisegment $\left\{ \Delta_1\cup \Delta_2 \right\}$ if $\Delta_1\cap \Delta_2$ is empty.

Suppose $\langle \Delta_1\cup \Delta_2 + \Delta_1\cap \Delta_2 \rangle \times \pi$ is not irreducible to arrive a contradiction. Then, we must have $\tau_1 \not\cong \tau_2$, which follows from  $\square$-irreducibility of $\langle \Delta_1\cup \Delta_2+\Delta_1\cap \Delta_2\rangle$ (\cite[Theorem 9.7]{Ze80} for $D=F$, see \cite[Lemma 2.5]{Ta90} and \cite[Lemma 6.17]{LM16} for general) and the SI property of $\langle \Delta_1\cup \Delta_2+\Delta_1\cap \Delta_2 \rangle$ \cite[Lemma 2.8]{LM19}. Thus, we must have either $\tau_1 \not\cong \tau$ or $\tau_2 \not\cong \tau$.

We now consider two cases separately:
\begin{enumerate}
\item $\tau_1 \not\cong \tau$. Then $\tau_1$ appears in the submodule of $\langle \Delta_1\cup \Delta_2+\Delta_1\cap \Delta_2 \rangle \times \pi$ and so appears in the submodule of $\langle \Delta_2 \rangle \times \langle \Delta_1 \rangle \times \pi$. Using (\ref{eqn first appear in quotient}), we also have that $\tau_1$ is a submodule of $\pi \times \langle \Delta_2 \rangle \times \langle \Delta_1\rangle$. Since $\tau_1 \not\cong \tau$, $\tau_1$ must come from the submodule of $\pi \times \langle \Delta_1\cup \Delta_2+\Delta_1\cap \Delta_2\rangle$. This shows that 
\[   \tau_1\cong \mathrm{soc}(\pi \times \langle \Delta_1\cup \Delta_2+\Delta_1\cap \Delta_2 \rangle )\cong \mathrm{soc}(\langle \Delta_1\cup \Delta_2+\Delta_1\cap \Delta_2\rangle\times \pi) .\]
In other words, the socle and cosocle of $\pi \times \langle \Delta_1\cup \Delta_2+\Delta_1\cap \Delta_2 \rangle$ coincides by \cite[Corollary 2.4]{LM19}. By \cite[Lemma 2.8]{LM19}, we must then have that $\pi \times \langle \Delta_1\cap \Delta_2+\Delta_1\cup \Delta_2\rangle$ is irreducible.
\item $\tau_2\not\cong \tau$. The proof is similar to the previous case, but we consider quotients rather than submodules and use (\ref{eqn first appear in submodule}) rather than (\ref{eqn first appear in quotient}).
\end{enumerate}
\end{proof}




\subsection{Some explicit criteria for a multisegment in $\mathcal M_{\pi}$}
A segment $\Delta=[a,b]_{\rho}$ is said to be {\it juxtaposed} to another segment $\Delta'=[a',b']_{\rho}$ if either $b+1=a'$ or $b'+1=a$.

\begin{remark}
\begin{enumerate}
\item Let $\pi \cong \mathrm{St}(\mathfrak m)$ for $\mathfrak m \in \mathrm{Mult}$ with all segments in $\mathfrak m$ mutually unlinked. Let $\mathfrak m \in \mathrm{Mult}$ such that $\pi \cong \mathrm{St}(\mathfrak m)$. Then $\mathfrak n \in \mathcal M_{\pi}$ if and only if any segment in $\mathfrak n$ is not juxtaposed to any segment in $\mathfrak m$. (See \cite[TH\'EOR\`EME 0.1]{BLM13})
\item Let $\pi$ be a Speh representation with the corresponding multisegment $\mathfrak m$. We label the segments in $\mathfrak m=\left\{ \Delta_1, \ldots, \Delta_r \right\}$ satisfying $\Delta_1<\ldots <\Delta_r$. Then $\mathfrak n \in \mathcal M_{\pi}$ if and only if any segment $\Delta$ in $\mathfrak n$ satisfies $\Delta \not< \Delta_1$ and $\Delta_r \not< \Delta$. (See \cite[Lemma 6.5]{LM16})
\end{enumerate}
\end{remark}

\section{Indecomposability under product functor: non-isomorphic cases} \label{ss indecomp noniso case}

\subsection{Some results on irreducibility}

In this section, we recall some results on the irreducibility of parabolic inductions. Most results are from or deduced from \cite{LM16}.

\begin{lemma} \cite[Lemma 3.9]{LM16} \label{lem commut of induction}
Let $\Delta \in \mathrm{Seg}$ and let $\mathfrak m \in \mathrm{Mult}$. Then $\langle \Delta \rangle \times \langle \mathfrak m \rangle$ is irreducible if and only if $\langle \Delta\rangle \times \langle \mathfrak m \rangle \cong \langle \mathfrak m \rangle \times \langle \Delta \rangle$.  
\end{lemma}

\begin{proposition} \label{prop segment indecomp embed} \cite[Proposition 6.1]{LM16}
Let $\mathfrak m \in \mathrm{Mult}$ and let $\pi=\langle \mathfrak m \rangle$. Let $\mathfrak p \in \mathcal M_{\pi}$. Then 
\begin{enumerate}
\item $\langle \mathfrak p \rangle \times \pi$ is irreducible; and
\item $\zeta(\mathfrak p)\times \pi \hookrightarrow \zeta(\mathfrak p+\mathfrak m)$; and
\item $\pi \times \zeta(\mathfrak p) \hookrightarrow \zeta(\mathfrak p+\mathfrak m)$. 
\end{enumerate} 
\end{proposition}

\begin{lemma}  \label{lem dual of lc condition}
Let $\mathfrak m, \mathfrak p \in \mathrm{Mult}$. Then $\mathfrak p \in \mathcal M_{\langle \mathfrak m \rangle}$ if and only if $\mathfrak p^{\vee} \in \mathcal M_{\langle \mathfrak m^{\vee} \rangle}$. 
\end{lemma}

\begin{proof}
This follows from definitions.
\end{proof}

\subsection{Indecomposability} \label{ss indecomposable repn}


We remark that an analogous result holds for other connected reductive groups with replacing the Zelevinsky classification by the Langlands classification (also see \cite{Ch18}). For the Langlands classification version, we remark that there is also an analogous statement for branching laws \cite{Ch21+}, with the generic case conjectured by D. Prasad \cite{Pr18} and proved in \cite{CS21} by Savin and the author.

\begin{lemma} \label{lem compare ext vanish}
Let $\mathfrak m, \mathfrak n \in \mathrm{Mult}_n$. Suppose $\mathfrak n \neq \mathfrak m$. Then 
\[  \mathrm{Ext}^i_{G_n}(\zeta(\mathfrak m^{\vee})^{\vee}, \zeta( \mathfrak n))=0
\]
for all $i$.
\end{lemma}

\begin{proof}

We shall prove by an induction on the sum of the numbers of segments in $\mathfrak m$ and $\mathfrak n$. When both $\mathfrak m$ and $\mathfrak n$ are empty sets, there is nothing to prove.

Let $\mathfrak m=\left\{ \Delta_1, \ldots, \Delta_r \right\}$ with 
\[  b(\Delta_1) \not< \ldots \not< b(\Delta_r) .\]
Similarly, let $\mathfrak n=\left\{ \Delta_1', \ldots, \Delta_s' \right\}$ with
\[  b(\Delta_1') \not< \ldots \not< b(\Delta_s') . \]

If no segment in $\mathfrak n$ satisfies $b(\Delta_i') \geq  b(\Delta_1)$, then a cuspidal support argument gives that $\mathrm{Ext}^i_{G_n}(\zeta(\mathfrak m^{\vee})^{\vee}, \zeta(\mathfrak n))=0$ for all $i$. Furthermore, if we have $\Delta_i'$ such that $b(\Delta_i')>b(\Delta_1)$, then a cuspidal support argument again gives that 
\[ \mathrm{Ext}^k_{G_n}(\zeta(\mathfrak m^{\vee})^{\vee}, \zeta(\mathfrak n))=0. \] 

Set $\rho=b(\Delta_1)$. Thus, now we consider that $b(\Delta_i')\cong \rho$ and there is no segment $\Delta_j'$ in $\mathfrak n$ satisfying $b(\Delta_j')>\rho$. Now, by relabelling if necessary (using Lemma \ref{lem commut of induction}), we may assume that $\Delta_1$ is a shortest segment in $\mathfrak m$ with $b(\Delta_1)\cong \rho$, and similarly assume that $\Delta_1'$ is a shortest segment in $\mathfrak n$ with $b(\Delta_1')\cong \rho$. We now consider the following three cases:
\begin{itemize}
\item Suppose $\Delta_1' \subsetneq \Delta_1$. Then, Frobenius reciprocity gives that:
\[ (*)\quad  \mathrm{Ext}^i_{G_n}(\langle \Delta_r \rangle \times \ldots \times \langle \Delta_1\rangle, \langle \Delta_1 \rangle \boxtimes \zeta(\mathfrak n-\Delta_1')) \cong \mathrm{Ext}^i_{G_n}((\langle \Delta_r \rangle \times \ldots \times \langle \Delta_1\rangle)_{N_{n-l_{abs}(\Delta_1')}}, \langle\Delta_1' \rangle\boxtimes \zeta(\mathfrak n-\Delta_1')) .
\]
Now one analyzes the layers from the geometric lemma on the term $(\langle \Delta_r \rangle \times \ldots \times \langle \Delta_1\rangle))_{N_{n-l_{abs}(\Delta_1')}}$ (also see Section \ref{ss jacquet segment steinberg}). One sees that no layer has the same cuspidal support as $\langle \Delta_1' \rangle \boxtimes \zeta(\mathfrak n-\Delta_1')$, and so this gives such desired Ext-vanishing by the standard argument on an action of the Bernstein center.
\item Suppose $\Delta_1 \subsetneq \Delta_1'$. Then, one uses that
\[  \mathrm{Ext}^i_{G_n}(\zeta(\mathfrak m^{\vee})^{\vee}, \zeta(\mathfrak n)) \cong \mathrm{Ext}^i_{G_n}(\zeta(\mathfrak n)^{\vee}, \zeta(\mathfrak m^{\vee})) .
\]
Now, we write $\zeta(\mathfrak m^{\vee})\cong \zeta(\mathfrak m^{\vee}-\Delta_1^{\vee})\times \langle \Delta_1^{\vee}\rangle$. One applies Frobenius reciprocity to give that:
\[ \mathrm{Ext}^i_{G_n}(\zeta(\mathfrak n)^{\vee}, \zeta(\mathfrak m^{\vee})) \cong  \mathrm{Ext}^i_{G_{n-l_{abs}(\Delta_1)}\times G_{l_{abs}(\Delta_1)}}((\langle \Delta_1'{}^{\vee}\rangle \times \ldots \times \langle \Delta_s'{}^{\vee} \rangle)_{N_{l_{abs}(\Delta_1)}}, \langle \Delta_r{}^{\vee} \rangle \times \ldots \times \langle \Delta_2{}^{\vee}\rangle \boxtimes \langle \Delta_1{}^{\vee} \rangle). 
\]
Now again analysing layers in the geometric lemma on $(\langle \Delta_1'{}^{\vee}\rangle \times \ldots \times \langle \Delta_s'{}^{\vee} \rangle)_{N_{l_{abs}(\Delta_1)}}$ (see Section \ref{ss jacquet segment steinberg} again), one can compare cuspidal supports to give Ext-vanishing.
\item Suppose $\Delta_1=\Delta_1'$. Then we apply the Frobenius reciprocity as (*). Then, again we compute the layers from the geometric lemma on the term
\[ (\langle \Delta_r \rangle \times \ldots \times \langle \Delta_1\rangle)_{N_{n-l_{abs}(\Delta_1')}}. \]
Then, a cuspidal support consideration on the $G_{l_{abs}(\Delta_1)}$ factor in $G_{l_{abs}(\Delta_1)}\times G_{n-l_{abs}(\Delta_1)}$ gives that only possible layers contributing a non-zero Ext-group take the form:
\[  \langle \Delta_1 \rangle \boxtimes (\langle \Delta_r\rangle \times \ldots \times \langle \Delta_2\rangle) .
\]
Let $G'=G_{l_{abs}(\Delta_1)}\times G_{n-l_{abs}(\Delta_1)}$. Now, by the K\"unneth formula, 
\begin{align*}
 &  \mathrm{Ext}^i_{G'}(\langle \Delta_1 \rangle \boxtimes (\langle \Delta_r\rangle \times \ldots \times \langle \Delta_2\rangle), \langle \Delta_1' \rangle \boxtimes \zeta(\mathfrak n-\Delta_1')) \\
= & \bigoplus_{k+l=i} \mathrm{Ext}^k_{G'}(\langle \Delta_1 \rangle , \langle \Delta_1' \rangle )\boxtimes  \mathrm{Ext}^l_{G'}(\langle \Delta_r \rangle \times \ldots \times \langle \Delta_2 \rangle, \zeta(\mathfrak n-\Delta_1')) 
\end{align*}
The latter term is zero by the induction, and so such layer will also give vanishing Ext-groups. Now, since all layers in the geometric lemma give vanishing Ext-groups, we again have that $\mathrm{Ext}^i_{G_n}(\zeta(\mathfrak m^{\vee})^{\vee}, \zeta(\mathfrak n))=0$ for all $i$. 
\end{itemize}

\end{proof}

\begin{lemma} \label{lem ext vanishing under zel order}
Let $\mathfrak m, \mathfrak n \in \mathrm{Mult}_n$. Suppose $\mathfrak n \not\leq_Z  \mathfrak m$. Then 
\[  \mathrm{Ext}^i_{G_n}(\langle \mathfrak m \rangle, \zeta(\mathfrak n)) =0
\]
for all $i$.
\end{lemma}

\begin{proof}
The basic case is that when all the segments in $\mathfrak m$ are unlinked. In such case, either $\zeta(\mathfrak n)$ does not have the same cuspidal support as $\langle \mathfrak m \rangle$; or $\mathfrak n$ is not generic. That case then follows from Lemma \ref{lem compare ext vanish}. 

We now consider that some segments in $\mathfrak m$ are unlinked. Then it admits a short exact sequence:
\[   0 \rightarrow  \omega  \rightarrow  \zeta(\mathfrak m^{\vee})^{\vee}  \rightarrow  \langle \mathfrak m \rangle \rightarrow 0 ,
\]
where $\omega$ is the kernel of the surjection. Then, the Zelevinsky theory \cite[Theorem 7.1]{Ze80} implies that any simple composition factor of $\omega$ has the associated multisegment $\mathfrak m'$ with $\mathfrak m' \leq_Z \mathfrak m$. Thus we still have $\mathfrak m' \not\leq_Z \mathfrak n$. Inductively on $\leq_Z$ (the basic case explained above), we have that:
\[  \mathrm{Ext}^i_{G_n}(\omega, \zeta(\mathfrak n) )=0
\]
for all $i$. Thus a long exact sequence argument gives that
\[  \mathrm{Ext}^i_{G_n}(\zeta(\mathfrak m^{\vee})^{\vee}, \zeta(\mathfrak n))\cong \mathrm{Ext}^i_{G_n}(\langle \mathfrak m \rangle, \zeta(\mathfrak n)) .
\]
Now the former one is zero by Lemma \ref{lem compare ext vanish} and so the latter one is also zero.
\end{proof}

For $\pi_1, \pi_2 \in \mathrm{Irr}$, we write $\pi_1 \leq_Z \pi_2$ if $\mathfrak m_1 \leq_Z \mathfrak m_2$, where $\mathfrak m_1$ and $\mathfrak m_2$ are the unique multisegments such that $\pi_1\cong \langle \mathfrak m_1 \rangle$ and $\pi_2 \cong \langle \mathfrak m_2 \rangle$.

\begin{lemma} \label{lem embedding irr to zel std}
Let $\lambda$ be a representation of $G_n$ of length $2$. Suppose $\lambda$ is indecomposable and the two simple composition factors of $\lambda$ are not isomorphic. Then either
\begin{enumerate}
\item $\lambda \hookrightarrow \zeta(\mathfrak p)$ for some multisegment $\mathfrak p$; or
\item $\lambda^{\vee} \hookrightarrow \zeta(\mathfrak p)$ for some multisegment $\mathfrak p$.
\end{enumerate}
\end{lemma}

\begin{proof}
Let $\pi$ be the simple quotient of $\lambda$ and let $\pi'$ be the simple submodule of $\lambda$. We consider the following three cases:
\begin{itemize}
\item Case 1: $\pi <_Z \pi'$. Let $\mathfrak p$ be the multisegment such that $\pi' \cong \langle \mathfrak p \rangle$. We have the following short exact sequence:
\[   0 \rightarrow \pi' \rightarrow \lambda \rightarrow \pi \rightarrow 0 . \]
Then applying $\mathrm{Hom}_{G_n}(., \zeta(\mathfrak p))$, we have the following long exact sequence:
\[   0 \rightarrow \mathrm{Hom}_{G_n}(\pi, \zeta(\mathfrak p)) \rightarrow \mathrm{Hom}_{G_n}(\lambda, \zeta(\mathfrak p)) \rightarrow \mathrm{Hom}_{G_n}(\pi', \zeta(\mathfrak p)) \rightarrow \mathrm{Ext}^1_{G_n}(\pi, \zeta(\mathfrak p)) .
\]
By Lemma \ref{lem ext vanishing under zel order}, the first and last terms are zero, and the third term has one-dimensional. Thus the unique map from $\lambda$ to $\zeta(\mathfrak p)$ is still non-zero when restricting to $\pi'$. Since $\pi'$ is the unique simple module, the map must then be an embedding. 
\item Case 2: $\pi' <_Z \pi$. In such case, we consider $\lambda^{\vee}$, which has simple submodule $\pi'{}^{\vee}$ and simple quotient $\pi^{\vee}$. We still have that $\pi'{}^{\vee} <_Z \pi^{\vee}$. Now the argument in Case 1 gives the embedding $\lambda^{\vee} \hookrightarrow \zeta(\mathfrak p)$ for some multisegment $\mathfrak p$. 
\item Case 3: $\pi'$ and $\pi$ are not $\leq_Z$-comparable. It suffices to prove that 
\[ \mathrm{Ext}^1_{G_n}(\pi, \pi')=0 \]
 i.e. such indecomposable $\lambda$ does not happen. To this end, let $\mathfrak p$ and $\mathfrak p'$ be the multisegments such that $\pi \cong \langle \mathfrak p \rangle$ and $\pi' \cong \langle \mathfrak p' \rangle$. We consider the following short exact sequences:
\[  0 \rightarrow \langle \mathfrak p' \rangle \rightarrow \zeta(\mathfrak p') \rightarrow \omega \rightarrow 0,
\]
where $\omega$ is the cokernel of the first injection. Then, a long exact sequence argument with Lemma \ref{lem ext vanishing under zel order} gives 
\[  \mathrm{Hom}_{G_n}(\langle \mathfrak p \rangle, \omega) \cong \mathrm{Ext}^1_{G_n}(\langle \mathfrak p \rangle, \langle \mathfrak p' \rangle) .
\]
The former one is zero since any simple composition factor $\omega'$ in $\omega$ also satisfies $ \langle \mathfrak p \rangle \not\leq_Z \omega'$. Thus the latter Ext is also zero.
\end{itemize}
\end{proof}

 Define
\[ \mathrm{Alg}_{\pi}(G_n) 
\]
to be the full subcategory of $\mathrm{Alg}_f(G_n)$ of objects, all of whose simple composition factors are isomorphic to $\langle \mathfrak m\rangle$ for some $\mathfrak m \in \mathcal M_{\pi}$. In other words, $\mathrm{Alg}_{\pi}(G_n)$ is the full Serre subcategory generated by simple objects of the form $\langle \mathfrak m \rangle$ for $\mathfrak m$ in $\mathcal M_{\pi}$.

\begin{proposition} \label{prop indecomp non-isomorphic}
Let $\lambda$ be a representation of $G_n$ of length $2$. Suppose $\lambda$ is indecomposable. Suppose the two simple composition factors of $\lambda$ are not isomorphic and both are in $\mathrm{Alg}_{\pi}(G_n)$. Then $\pi \times \lambda$ is still an indecomposable representation of length $2$.
\end{proposition}

\begin{proof}
By Proposition \ref{prop segment indecomp embed}(1), we have that $\pi \times \lambda$ has length $2$. To show the indecomposability, it suffices to show that $\pi \times \lambda$ has either unique simple quotient or unique simple submodule. Let $\pi_1$ be the simple quotient of $\lambda$ and let $\pi_2$ be the simple submodule of $\lambda$. Let $\mathfrak m$ be the multisegment associated to $\pi$.

According to the proof of Lemma \ref{lem embedding irr to zel std}, we must have one of the following two cases:
\begin{itemize}
 \item Case (1): $\pi_1 <_Z \pi_2$. In such case, there exists an embedding, by Lemma \ref{lem embedding irr to zel std},  
\[   \lambda \hookrightarrow \zeta(\mathfrak p)
\]
for some multisegment $\mathfrak p$. Thus we also have an embedding:
\[  \pi \times \lambda \hookrightarrow \pi \times \zeta(\mathfrak p). \]
But the latter module embeds to $\zeta(\mathfrak m+\mathfrak p)$ by Proposition \ref{prop segment indecomp embed}(3), which has a unique submodule. Thus, $\pi \times \lambda$ also has unique submodule and so is indecomposable.
 \item Case (2): $\pi_2 <_Z \pi_1$. It suffices to show that
\[  (\pi \times \lambda)^{\vee} =\pi^{\vee} \times \lambda^{\vee}
\]
has unique simple submodule. We have the embedding, by (the proof of) Lemma \ref{lem embedding irr to zel std} again:
\begin{align} \label{eqn second embedding in indecomp}
 \lambda^{\vee} \hookrightarrow \zeta(\mathfrak q)
\end{align}
for some $\mathfrak q \in \mathrm{Mult}$. We now consider the following embeddings:
\[  \pi^{\vee} \times \lambda^{\vee} \hookrightarrow \pi^{\vee}\times \zeta(\mathfrak q)  \hookrightarrow \zeta(\mathfrak q+\mathfrak m^{\vee}) ,
\]
where the first embedding follows from (\ref{eqn second embedding in indecomp}) and the second embedding follows from Proposition \ref{prop segment indecomp embed} and Lemma \ref{lem dual of lc condition}. 
\end{itemize}

\end{proof}

\section{Some results involving the geometric lemma}

\subsection{A counting problem}

In order to give a favour of using the geometric lemma below, let us first consider the following lemma involving some counting arguments. We first define some notions.

For $\mathfrak m \in \mathrm{Mult}$ and $\Delta \in \mathrm{Seg}$, let 
\[ \mathfrak m_{\Delta} = \{\overbrace{\Delta, \ldots, \Delta}^{k \mbox{ times }} \},
\]
where $k$ is the multiplicity of $\Delta$ in $\mathfrak m$. In particular, $\mathfrak m_{\Delta}$ is a submultisegment of $\mathfrak m$. For example, if $\mathfrak m=\left\{ [1], [1,2],[1,2], [2,3], [2,3], [4] \right\}$, then $\mathfrak m_{[1,2]}=\left\{ [1,2], [1,2] \right\}$ and $\mathfrak m_{[4]}=\left\{ [4] \right\}$. 

For two segments $\Delta, \Delta'$, we write $\Delta <_b \Delta'$ if either one of the following conditions holds:
\begin{itemize}
\item $b(\Delta)< b(\Delta')$; or 
\item $b(\Delta)\cong b(\Delta')$ and $a(\Delta) \leq a(\Delta')$. 
\end{itemize}
We write $\Delta \leq_b \Delta'$ if $\Delta =\Delta'$ or $\Delta<_b \Delta'$. This defines a partial ordering on $\mathrm{Seg}$.

\begin{lemma} \label{lem combinatorial cuspidal repn}
Let $\mathfrak m \in \mathrm{Mult}$. We write $\mathfrak m=\left\{ \Delta_1, \ldots , \Delta_r \right\}$. Let $\Delta=[a,b]_{\rho}$ be a $\leq_b$-maximal element in $\mathfrak m$. For each segment $\Delta_i=[a_i,b_i]_{\rho_i}$, we write: $\Delta_i^{+}=[a_i,c_i]_{\rho_i}$ and $\Delta_i^{-}=[c_i+1, b_i]_{\rho_i}$ for some $a_i-1\leq c_i\leq b_i$. By abuse of notations, we write $A=\cup_i \Delta_i^{+} $ as a multiset of cuspidal representations. Let $k$ be the number of segments in $\mathfrak m_{\Delta}$. If
\[   A =\cup_{j=1}^k \Delta
\]
as multisets, then $c_i=b_i$ if $\Delta_i=\Delta$ and $c_i=a_i-1$ if $\Delta_i\neq \Delta$.
\end{lemma}

\begin{proof}
Note that there are $k$ copies of $b(\Delta)$ in $\cup_{j=1}^k\Delta$. Hence, we must also have $k$ copies of $b(\Delta)$ in $A$. Thus, we must have $k$-copies of $\Delta_i$ in $\mathfrak m_{b=b(\Delta)}$ such that $\Delta_i^+=\Delta_i$. We write such $k$ segments as $\Delta_{i_1}, \ldots, \Delta_{i_k}$. Now, recall that $\Delta$ is $\leq_b$-maximal from our choice, and so if one of $\Delta_{i_j} \neq \Delta$, then $\Delta_{i_j}$ contains the cuspidal representation $\nu^{-1}_{\rho}a(\Delta)$. Thus it is impossible. Hence, all $\Delta_{i_j}=\Delta$. Then a simple count gives that $A =\cup_{j=1}^k \Delta$.
\end{proof}

We now study some applications on the above Lemma \ref{lem combinatorial cuspidal repn}. For notational simplicity, for $\mathfrak m \in \mathrm{Mult}$, we set 
\[  \widetilde{\zeta}(\mathfrak m) = \zeta(\mathfrak m^{\vee})^{\vee} .\]
This coincides with the notion $\widetilde{\zeta(\mathfrak m)}$ in Section \ref{ss left right induction} when $D=F$.

\begin{lemma} \label{lem geometric lemma on standard ones}
Let $\mathfrak m_1, \mathfrak m_2 \in \mathrm{Mult}$. Let $\mathfrak m=\mathfrak m_1+\mathfrak m_2$. Let $\Delta$ be a $\leq_b$-maximal element in $\mathfrak m$. Let $n_1=l_{abs}(\mathfrak m_1)$, $n_2=l_{abs}(\mathfrak m_2)$, $i_1=n_1-l_{abs}((\mathfrak m_1)_{\Delta})$ and let $i_2=n_2-l_{abs}((\mathfrak m_2)_{\Delta})$. Let $i=i_1+i_2$. We now consider the filtration for $(\widetilde{\zeta}(\mathfrak m_1) \times \widetilde{\zeta}(\mathfrak m_2))_{N_i}$ from the geometric lemma in Section \ref{ss geo lemma}. The only layer from that filtration, which has the same cuspidal support as $\langle \mathfrak m_{\Delta} \rangle \boxtimes \langle \mathfrak m-\mathfrak m_{\Delta}\rangle$, takes the form 
\[   (*)\quad  \mathrm{Ind}^{G'}_P (\widetilde{\zeta}(\mathfrak m_1)_{N_{i_1}} \boxtimes \widetilde{\zeta}(\mathfrak m_2)_{N_{i_2}})^{\phi} ,
\]
where $G'=G_{n_1+n_2-i_1-i_2} \times G_{i_1+i_2}$ and $\phi: G_{n_1-i_1}\times G_{i_1}\times G_{n_2-i_2}\times G_{i_2} \rightarrow G_{n_1-i_1}\times G_{n_2-i_2}\times G_{i_1}\times G_{i_2}$.

Moreover, the component in $\widetilde{\zeta}(\mathfrak m_1)_{N_{i_1}}$ and $\widetilde{\zeta}(\mathfrak m_2)_{N_{i_2}}$ in (*) contributing to the factor $\langle \mathfrak m_{\Delta} \rangle \boxtimes \langle \mathfrak m -\mathfrak m_{\Delta} \rangle$ can be refined to 
\[ \langle (\mathfrak m_1)_{\Delta} \rangle \boxtimes \widetilde{\zeta}(\mathfrak m_1-(\mathfrak m_1)_{\Delta}), \quad \langle (\mathfrak m_2)_{\Delta} \rangle \boxtimes \widetilde{\zeta}(\mathfrak m_1-(\mathfrak m_2)_{\Delta} \rangle .
\]
\end{lemma}

\begin{proof}
The problem on the layer can be transferred to the counting problem by using the Jacquet functor computations in Section \ref{ss jacquet segment steinberg}. Then the lemma follows from Lemma \ref{lem combinatorial cuspidal repn}.
\end{proof}

\subsection{A direct summand computation}

\begin{lemma} \label{lem variation of cuspidal supprot}
Let $\mathfrak m \in \mathrm{Mult}$. Let $\Delta$ be a $\leq_b$-maximal element. Let $i=l_{abs}(\mathfrak m)-l_{abs}(\mathfrak m_{\Delta})$. Then the direct summand in $\langle \mathfrak m \rangle_{N_i}$ with same cuspidal support as $\langle \mathfrak m_{\Delta}\rangle \boxtimes \langle \mathfrak m-\mathfrak m_{\Delta}\rangle$ is actually isomorphic to $\langle \mathfrak m_{\Delta} \rangle \boxtimes \langle \mathfrak m-\mathfrak m_{\Delta} \rangle$. 
\end{lemma}

\begin{proof}
Let $\pi =\langle \mathfrak m \rangle$. Then, from standard results of the Zelevinsky classification \cite{Ze80}, 
\[    \pi \hookrightarrow  \langle \mathfrak m_{\Delta} \rangle \times \langle \mathfrak m -\mathfrak m_{\Delta} \rangle .
\]
This implies
\[   \pi_{N_i} \hookrightarrow (\langle \mathfrak m_{\Delta} \rangle \times \langle \mathfrak m-\mathfrak m_{\Delta} \rangle)_{N_i} .\]
Thus, any simple composition factor in $\pi_{N_i}$ appears as a composition factor in:
\[    (\tau_1 \times \tau_2) \boxtimes (\tau_3 \times \tau_4) 
\]
such that $\tau_1\boxtimes \tau_3$ is a simple composition factor in $\langle \mathfrak m_{\Delta} \rangle_{N_{i'}}$ and $\tau_2\boxtimes \tau_4$ is a simple composition factor in $\langle \mathfrak m-\mathfrak m_{\Delta} \rangle_{N_{i''}}$, where $i'+i''=i$. 

Let $k=|\mathfrak m_{\Delta}|$ and write $\Delta=[a,b]_{\rho}$. Then $\mathrm{csupp}(\tau_1) \cup \mathrm{csupp}(\tau_3)$ (union as a multiset) has $k$ number of $b(\Delta)$. Now we suppose some $b(\Delta)$ come from $\mathrm{csupp}(\tau_3)$ to obtain a contradiction. In such case, $\tau_3\boxtimes \tau_4$ also appears in $\zeta( \mathfrak m-\mathfrak m_{\Delta})_{N_{i''}}$. But the latter term can be computed from the geometric lemma again. One sees that if $\mathrm{csupp}(\tau_3)$ contains $b(\Delta)$, then it contains all the cuspidal representation in a segment $\Delta' \in \mathfrak m_{b=b(\Delta)}-\mathfrak m_{\Delta}$. Since $\Delta$ is $\leq_b$-maximal, $\Delta'$ contains $\nu_{\rho}\cdot a(\Delta)$. This gives a contradiction that $\mathrm{csupp}(\tau_1)\cup \mathrm{csupp}(\tau_3)=\mathrm{csupp}(\langle \mathfrak m_{\Delta}\rangle)$. 

We have concluded that all $b(\Delta)$ in $\mathrm{csupp}(\langle \mathfrak m_{\Delta}\rangle)$ arises from $\mathrm{csupp}(\tau_1)$. Then, we must have that $\tau_1=\langle \mathfrak m_{\Delta} \rangle$ and $i'=i$ and $i''=0$. This shows that the only layer in the geometric lemma giving the desired module is $\langle \mathfrak m_{\Delta} \rangle \boxtimes \langle \mathfrak m-\mathfrak m_{\Delta} \rangle$. This shows the lemma.
\end{proof}

\subsection{A refined computation}

\begin{lemma} \label{lem contributing factor of geometric lemma}
We use the notations in Lemma \ref{lem geometric lemma on standard ones}. We consider the filtration for $(\langle \mathfrak m_1 \rangle \times \langle \mathfrak m_2 \rangle)_{N_i}$ from the geometric lemma. The only layer from that filtration, which has the same cuspidal support as $\langle \mathfrak m_{\Delta} \rangle \boxtimes \langle \mathfrak m-\mathfrak m_{\Delta}\rangle$, takes the form 
\[   (**) \quad  \mathrm{Ind}^{G'}_P (\langle \mathfrak m_1 \rangle_{N_{i_1}} \boxtimes \langle \mathfrak m_2 \rangle_{N_{i_2}})^{\phi} .
\]
Moreover, the component in $\langle \mathfrak m_1 \rangle_{N_{i_1}}$ and $\langle \mathfrak m_2 \rangle_{N_{i_2}}$ in (**) contributing to the factor $\langle \mathfrak m_{\Delta} \rangle \boxtimes \langle \mathfrak m -\mathfrak m_{\Delta} \rangle$ can be refined to 
\[ \langle (\mathfrak m_1)_{\Delta} \rangle \boxtimes \langle \mathfrak m_1-(\mathfrak m_1)_{\Delta} \rangle, \quad \langle (\mathfrak m_2)_{\Delta} \rangle \boxtimes \langle \mathfrak m_1-(\mathfrak m_2)_{\Delta} \rangle .
\]
\end{lemma}

\begin{proof}
Note that the geometric lemma is functorial and so the first assertion follows from the corresponding one in Lemma \ref{lem geometric lemma on standard ones}. The second assertion then follows from Lemma \ref{lem variation of cuspidal supprot}.
\end{proof}

In the following applications, we shall need two modifications. One is to use Lemma \ref{lem contributing factor of geometric lemma} repeatedly while another one is to replace $\langle \mathfrak m_2 \rangle$ with an indecomposable module of length $2$. We shall avoid notation complications to give such precise statements and the meaning will become clearer when one sees the required statements in the following proofs.

\section{Constructing self-extensions}

For $\pi$ in $\mathrm{Irr}(G_n)$, we first show that self-extensions of $\pi$ can be constructed via self-extensions of its associated Zelevinsky standard module. Then we study self-extensions of Zelevinsky standard modules and show it can be reduced to a tempered case via a categorical equivalence in Corollary \ref{cor equivalence of categories}.

\subsection{Constructing extensions from $\zeta(\mathfrak m)$} \label{ss self extension}

Let $\mathfrak m \in \mathrm{Mult}$. Let $\pi=\langle \mathfrak m \rangle$. In this subsection, we explain how to construct extensions between two copies of $\langle \mathfrak m \rangle$ from extensions of two copies of $\zeta(\mathfrak m)$. One may compare with the study in \cite[Section 3]{Ch18}. We first show that one can do that by showing Lemma \ref{lem embedding ext groups} and then reinterpret the result via the Yoneda extension lemma.

\begin{lemma} \label{lem embedding ext groups}
Let $\mathfrak m \in \mathrm{Mult}_n$. Then we have a natural embedding
\[ \mathrm{Ext}^1_{G_n}(\langle \mathfrak m \rangle, \langle \mathfrak m \rangle) \hookrightarrow \mathrm{Ext}^1_{G_n}(\langle \mathfrak m \rangle, \zeta(\mathfrak m))\cong \mathrm{Ext}^1_{G_n}(\zeta(\mathfrak m), \zeta(\mathfrak m)) .
\]
\end{lemma}
\begin{proof}
We have 
\[ 0\rightarrow  \langle \mathfrak m \rangle \hookrightarrow \zeta(\mathfrak m) \rightarrow K \rightarrow 0 ,
\]
where $K$ is the cokernel of the first embedding.

Now, by Lemma \ref{lem compare ext vanish}, we have that, for all $i$, 
\[ \mathrm{Ext}^i_{G_n}(K, \zeta(\mathfrak m)) =0 . \]
Thus a standard long exact sequence gives that
\begin{align}\label{eqn ext iso}
  \mathrm{Ext}^i_{G_n}(\langle \mathfrak m \rangle , \zeta(\mathfrak m)) \cong \mathrm{Ext}^i_{G_n}(\zeta(\mathfrak m), \zeta(\mathfrak m)) .
\end{align}

Long exact sequence now gives that
\[ 0= \mathrm{Hom}_{G_n}(\langle \mathfrak m \rangle, K) \rightarrow \mathrm{Ext}^1_{G_n}(\langle \mathfrak m \rangle, \langle \mathfrak m \rangle ) \rightarrow \mathrm{Ext}^1_{G_n}(\langle \mathfrak m \rangle, \zeta(\mathfrak m) )  
\]
Thus, combining with the above isomorphism, 
\[  0 \rightarrow \mathrm{Ext}^1_{G_n}(\langle \mathfrak m \rangle, \langle \mathfrak m \rangle) \hookrightarrow \mathrm{Ext}^1_{G_n}(\zeta(\mathfrak m), \zeta(\mathfrak m)) 
\]
\end{proof}

We remark that the injection in Lemma \ref{lem embedding ext groups} is not an isomorphism in general. For example, if one takes $\mathfrak m=\left\{ [0], [1] \right\}$, then $\langle \mathfrak m \rangle \cong \mathrm{St}([0,1])$ and so $\mathrm{dim}~\mathrm{Ext}^1_{G_2}(\langle \mathfrak m \rangle, \langle \mathfrak m \rangle) =1$, but $\mathrm{dim}~\mathrm{Ext}^1_{G_2}(\zeta(\mathfrak m), \zeta(\mathfrak m))=2$.

We now explain Lemma \ref{lem embedding ext groups} in module language via the Yoneda extension interpretation (\cite[Ch III Theorem 9.1]{Ma63}, also see \cite[Section 6, Pages 71 and 83]{Ma63}). We can interpret an element in $\mathrm{Ext}_{G_n}^1(\langle \mathfrak m \rangle, \langle \mathfrak m \rangle)$ as a short exact sequence:
\[  0 \rightarrow \langle \mathfrak m \rangle \rightarrow \pi \rightarrow \langle \mathfrak m \rangle \rightarrow 0 .
\]
By using Lemma \ref{lem embedding ext groups}, there exist short exact sequences such that the following diagram commutes:
\[ \xymatrix{ 0 \ar[r] & \langle \mathfrak m \rangle \ar[r] \ar[d] & \pi \ar[r] \ar[d] &  \langle \mathfrak m \rangle \ar[r] \ar@{=}[d] & 0 \\ 
0 \ar[r] &\zeta(\mathfrak m ) \ar[r] \ar@{=}[d] & \pi'  \ar[r] \ar[d] & \langle \mathfrak m \rangle \ar[r] \ar[d] & 0 \\ 
0 \ar[r] & \zeta( \mathfrak m) \ar[r] & \pi'' \ar[r] & \zeta( \mathfrak m) \ar[r] &  0 }
\]
Since the leftmost and rightmost vertical maps are injections, the middle vertical maps are also injective. In other words, we obtain:

\begin{lemma} \label{lem embedding indecomposable to std}
Let $\pi$ be an indecomposable representation of $G_n$ of length two with both irreducible composition factors isomorphic to $\langle \mathfrak m \rangle$ for some $\mathfrak m\in \mathrm{Mult}_n$. Then there exists an indecomposable representation $\pi''$ of $G_n$ which 
\begin{itemize} 
\item admits a short exact sequence:
\[  0 \rightarrow \zeta(\mathfrak m) \rightarrow \pi'' \rightarrow \zeta(\mathfrak m) \rightarrow 0
\]
 and 
\item $\pi$ embeds to $\pi''$.
\end{itemize}
\end{lemma}


\subsection{Extensions between two $\zeta(\mathfrak m)$ }


Let $\Delta_1, \ldots, \Delta_r$ be all the distinct segments such that $\mathfrak m_{\Delta_i} \neq \emptyset$ and label in the way that $\Delta_1 \not\leq_b \Delta_2 \not\leq_b \ldots \not\leq_b \Delta_r$. Let $n_i=l_{abs}(\mathfrak m_{\Delta_i})$ for $i=1, \ldots, r$. Denote by $G(\mathfrak m)$ the group $G_{n_1}\times \ldots \times G_{n_r}$. Let, as a $G(\mathfrak m)$-representation, 
\[  [ \mathfrak m ] =  \langle \mathfrak m_{\Delta_1} \rangle \boxtimes \ldots \boxtimes \langle \mathfrak m_{\Delta_r} \rangle .
\]

\begin{lemma} \label{lem ext group self isomor}
For $\mathfrak m \in \mathrm{Mult}_n$, and for any $i$,
\[  \mathrm{Ext}^i_{G_n}(\zeta(\mathfrak m), \zeta(\mathfrak m)) \cong \mathrm{Ext}^i_{G(\mathfrak m)}([ \mathfrak m ], [ \mathfrak m ]) .
\]
\end{lemma}

\begin{proof}
Let $\Delta$ be a $\leq_b$-maximal element such that $\mathfrak m_{\Delta} \neq \emptyset$. Then we write 
\[  \zeta(\mathfrak m)=\langle \mathfrak m_{\Delta}\rangle \times \zeta(\mathfrak m-\mathfrak m_{\Delta} ) .  
\]
Let $n_1=l_{abs}(\mathfrak m_{b= \rho'})$ and let $n=l_{abs}(\mathfrak m)$.
Now, 
\[  \mathrm{Ext}^i_{G_n}(\zeta(\mathfrak m), \zeta(\mathfrak m)) \cong \mathrm{Ext}^i_{G_{n_1}\times G_{n-n_1}}(\langle \mathfrak m_{\Delta} \rangle \boxtimes \zeta(\mathfrak m-\mathfrak m_{\Delta}), \langle \mathfrak m_{\Delta} \rangle \boxtimes \zeta(\mathfrak m-\mathfrak m_{\Delta})) ,
\]
which follows from first applying Frobenius reciprocity and then using the geometric lemma and Lemma \ref{lem combinatorial cuspidal repn} to single out the only layer that has the same cuspidal support as $\zeta(\mathfrak m_{\Delta})\boxtimes \zeta(\mathfrak m-\mathfrak m_{\Delta})$. We now repeat similar process for $\zeta(\mathfrak m-\mathfrak m_{\Delta})$.
\end{proof}


We now focus on $i=1$ in Lemma \ref{lem ext group self isomor} to study first extensions. We now have the following:

\begin{proposition} \label{prop isom levi and induced one}
Let $\pi_1, \pi_2 \in \mathrm{Alg}_f(G_n)$. Suppose each of $\pi_1$ and $\pi_2$ admits a filtration with successive subquotients isomorphic to $\zeta(\mathfrak m)$. Let $M=G(\mathfrak m)$ and let $P$ be the standard parabolic subgroup containing $M$. Then
\begin{enumerate}
\item for each $i=1,2$, there exists an admissible $M$-representation $\tau_i$ which admits a filtration with successive subquotients isomorphic to $[ \mathfrak m ]$ such that $\pi_i \cong \mathrm{Ind}_{P}^{G_n}\tau_i$, and 
\item $\pi_1 \cong \pi_2$ if and only if $\tau_1 \cong \tau_2$. 
\end{enumerate}
\end{proposition}


\begin{proof}
Let $n=l_{abs}(\mathfrak m)$. Let $P=MN$ be the Levi decomposition. We first consider $(\pi_i)_N$. Let $\tau_i$ be the projection of $(\pi_i)_N$ to the component that has the same cuspidal support as $[\mathfrak m ]$. By repeated use of Lemma \ref{lem geometric lemma on standard ones} (under the situation that $\mathfrak m_2$ in Lemma \ref{lem geometric lemma on standard ones} is empty), we have that $\tau_i$ admits a filtration whose successive subquotients are isomorphic to $[\mathfrak m]$. Then, applying Frobenius reciprocity, we have
\[  \mathrm{Hom}_{G(\mathfrak m)}((\pi_i)_N, \tau_i) \cong \mathrm{Hom}_{G_n}(\pi_i, \mathrm{Ind}_P^{G_n}\tau_i)
\]
and so we obtain a map $f$ in $\mathrm{Hom}_{G_n}(\pi_i, \mathrm{Ind}_P^{G_n}\tau_i)$ corresponding to the surjection 
\[   (\pi_i)_N  \twoheadrightarrow \tau_i .
\]

\noindent
{\it Claim:} $f$ is an isomorphism. 

\noindent
{\it Proof of the claim:} If $f$ is not an isomorphism, then by counting the number of composition factors, we must have an embedding $\langle \mathfrak m \rangle \stackrel{\iota}{\hookrightarrow} \pi_i$ such that $f \circ \iota =0$. However, via the functoriality of Forbenius reciprocity, we also have a composition of maps
\[  [ \mathfrak m ] \hookrightarrow \langle \mathfrak m\rangle_N \hookrightarrow (\pi_i)_N \twoheadrightarrow \tau
\] 
is zero. However, this is not possible since the multiplicity of $[ \mathfrak m ]$ in $(\pi_i)_N$ agrees with that in $\tau$ via the construction above. \\

Now the claim gives that $\pi_i \cong \mathrm{Ind}_P^{G_n} \tau_i$ and this proves (1). 

We now prove (2). The if direction is clear. For the only if direction, let $f: \mathrm{Ind}_P^{G_n}\tau_1 \rightarrow \mathrm{Ind}_P^{G_n}\tau_2$ be the isomorphism. Then the corresponding map, denoted $\widetilde{f}$, under Frobenius reciprocity is given by: $\widetilde{f}(x)=f(x)(1)$, where $1$ is the evaluation at the identity by viewing $f(x)$ as a function in $\mathrm{Ind}_P^{G_n}\tau_2$; and $x$ is any representative in $(\pi_1)_N$. Since $f$ is an isomorphism, the map $\widetilde{f}$ is surjective. Thus the multiplicity of $[ \mathfrak m ]$ in $\tau_1$ must be at least that in $\tau_2$. Similarly, we can obtain that the multiplicity of $[ \mathfrak m ]$ in $\tau_2$ must be at least that in $\tau_1$. Since the two multiplicities agree, $\widetilde{f}$ restricted to $\tau_1$ in $(\pi_1)_N$ must be an isomorphism.
\end{proof}

\begin{corollary} \label{cor equivalence of categories}
Let $\mathfrak m \in \mathrm{Mult}_n$. Let $\mathcal C_1$ be the full subcategory of $\mathrm{Alg}_f(G_n)$ whose objects admit a finite filtration with successive subquotients isomorphic to $\zeta(\mathfrak m)$. Let $\mathcal C_2$ be the full subcategory of $\mathrm{Alg}_f(G(\mathfrak m))$ whose objects admit a finite filtration with successive subquotients isomorphic to $[ \mathfrak m ]$. There is a categorical equivalence between $\mathcal C_1$ and $\mathcal C_2$. Here $\mathrm{Alg}_f(G(\mathfrak m))$ is the category of smooth representations of finite length of $G(\mathfrak m)$.
\end{corollary}

\begin{proof}
Using Proposition \ref{prop isom levi and induced one} (and its notations), one can write $\pi_i=\mathrm{Ind}_P^{G_n}\tau_i$ for some $\tau_i$ in $\mathcal C_1$. It remains to see that the induction functor in the previous proposition also defines an isomorphism on the morphism spaces. The induced map
\[  \mathrm{Hom}_{G(\mathfrak m)}(\tau_1, \tau_2) \rightarrow \mathrm{Hom}_{G_n}(\pi_1, \pi_2)
\]
is injective since the parabolic induction sends any non-zero objects to non-zero objects. Now it follows from Frobenius reciprocity,
\[  \mathrm{Hom}_{G_n}(\pi_1, \pi_2) \cong \mathrm{Hom}_{G(\mathfrak m)}((\pi_1)_N, \tau_2)\cong \mathrm{Hom}_{G(\mathfrak m)}(\tau_1, \tau_2) .
\]
The last isomorphism follows from the proof of Proposition \ref{prop isom levi and induced one} that $\tau_i$ is the component of $(\pi_i)_N$ that has the same cuspidal support as $\tau_i$. Now the finite-dimensionality of the Hom spaces implies that the injection is also an isomorphism, as desired.
\end{proof}

\begin{corollary}
We use the notations in Corollary \ref{cor equivalence of categories}. Let $\tau$ be an object in $\mathcal C_2$ and let $\pi$ be the corresponding representation under the equivalence in Corollary \ref{cor equivalence of categories}. Then 
\[  \mathrm{dim}~\mathrm{Hom}_{G_n}(\langle \mathfrak m \rangle, \pi) = \mathrm{dim}~\mathrm{Hom}_{G(\mathfrak m)}([ \mathfrak m ], \tau) .
\]
\end{corollary}

\begin{proof}
Let $k=\mathrm{dim}~\mathrm{Hom}_{G(\mathfrak m)}([ \mathfrak m ], \tau)$. By the equivalence of categories, we have an embedding:
\[  \overbrace{\zeta(\mathfrak m) \oplus \ldots \oplus \zeta(\mathfrak m)}^{k \mbox{ times} } \hookrightarrow  \pi=\mathrm{Ind}_P^{G_n}\tau   .\]
Hence, $\mathrm{dim}~\mathrm{Hom}_{G_n}(\langle \mathfrak m \rangle, \pi) \geq k$. 

Let $l=\mathrm{dim}~\mathrm{Hom}_{G_n}(\langle \mathfrak m \rangle, \pi)$. Suppose $l>k$. Then, we have an embedding:
\[   \overbrace{\langle \mathfrak m \rangle \oplus \ldots \oplus \langle \mathfrak m \rangle}^{l \mbox{ times} } \hookrightarrow \pi=\mathrm{Ind}_P^{G_n}\tau .
\]
Now since the Jacquet functor is exact, we have that:
\[   \overbrace{[ \mathfrak m ] \oplus \ldots \oplus [ \mathfrak m ]}^{l \mbox{ times} } \hookrightarrow \tau.
\]
This then gives a contradiction. Hence, we must have that $l=k$.
\end{proof}

\section{Big derivatives}

In this section, we introduce the notion of big derivatives and describe some basic properties.

\subsection{Big derivatives} \label{ss big derivatives}

\begin{definition} \label{def big derivatives}
Let $\sigma \in \mathrm{Irr}(G_i)$. For $\pi \in \mathrm{Alg}_f(G_n)$, define 
\[ \mathbb D_{\sigma}(\pi) = \mathrm{Hom}_{G_{i}}( \sigma, \pi_{N_{n-i}^-}),
\]
where $\pi_{N_{n-i}^-}$ is regarded as a $G_i$-representation via embedding $G_i$ to the first factor of $G_i \times G_{n-i}$. We regard $\mathbb D_{\sigma}(\pi)$ as a $G_{n-i}$-representation via:
\[   (g.f)(x)=(I_i, g).f(x) .
\]

 By applying the element $\begin{pmatrix} & I_{n-i} \\ I_i & \end{pmatrix}$, one can switch the $G_i \times G_{n-i}$-representation $\pi_{N_{n-i}^-}$ to $G_{n-i}\times G_i$-representation $\pi_{N_i}$. This gives the following isomorphism:
\begin{align}
 \mathbb D_{\sigma}(\pi) \cong \mathrm{Hom}_{G_{i}}(\sigma, \pi_{N_i}) .
\end{align}
We shall use the identification frequently in Section \ref{s application si derivatives}.

We similarly define the left version as:
\[ \mathbb D'_{\sigma}(\pi)=\mathrm{Hom}_{G_{i}}(\sigma, \pi_{N_{i}^-}),
\]
where $\pi_{N_{i}^-}$ is regarded as a $G_{n-i}\times G_i$-representation via the embedding $G_i \hookrightarrow G_{n-i}\times G_i$ into the second factor of $G_{n-i}\times G_i$. We remark that $\mathbb D_{\sigma}$ and $\mathbb D'_{\sigma}$ are left-exact, but not exact. 
\end{definition}

We only prove results for $\mathbb D$ and results for $\mathbb D'$ can be formulated and proved similarly (c.f. Section \ref{ss left right induction}). When $\pi$ is $\square$-irreducible, $\mathbb D_{\pi}(\tau)$ is either zero or has unique simple submodule \cite{KKKO15}. If $\mathbb D_{\pi}(\tau) \neq 0$, we shall denote by $D_{\pi}(\tau)$ the unique submodule. 



\subsection{Composition of big derivatives}

\begin{proposition} \label{prop composition of derivatives}
Let $\sigma_1,\ldots, \sigma_r \in \mathrm{Irr}$ such that $\sigma_1 \times \ldots \times \sigma_r$ is still irreducible. Then 
\[ \mathbb D_{\sigma_r}\circ \ldots \circ  \mathbb D_{\sigma_1}(\pi) \cong \mathbb D_{\sigma_1\times \ldots \times \sigma_r}(\pi) .
\]
\end{proposition}

\begin{proof}
For the given condition, we have that $\sigma_1 \times \ldots \times \sigma_s$ is still irreducible for $s\leq r$. Thus it reduces to $r=2$. Let $n_1=\mathrm{deg}(\sigma_1)$ and $n_2=\mathrm{deg}(\sigma_2)$. We have: for any $\tau \in \mathrm{Alg}_f(G_{n-n_1-n_2})$ and $\pi \in \mathrm{Alg}_f(G_n)$,
\begin{align*}
\mathrm{Hom}_{G_{n-n_1-n_2}}( \tau,  \mathbb D_{\sigma_2}\circ \mathbb D_{\sigma_1}(\pi)) 
 \cong &  \mathrm{Hom}_{G_{n_2}\times G_{n-n_1-n_2}}(\sigma_2\boxtimes \tau,\mathbb D_{\sigma_1}(\pi)_{N_{n_2,n-n_1-n_2}^-}) \\
	 \cong & \mathrm{Hom}_{G_{n-n_1}}(\sigma_2 \times \tau, \mathbb D_{\sigma_1}(\pi)) \\
 \cong  & \mathrm{Hom}_{G_{n_1}\times G_{n-n_1}}(\sigma_1 \boxtimes (\sigma_2\times \tau), \pi_{N^-_{n_1,n-n_1}}) \\
 \cong  & \mathrm{Hom}_{G_n}(\sigma_1 \times \sigma_2 \times \tau, \pi) \\
 \cong  & \mathrm{Hom}_{G_{n_1+n_2}\times G_{n-n_1-n_2}}((\sigma_1\times \sigma_2)\boxtimes \tau, \pi_{N_{n_1+n_2,n-n_1-n_2}}) \\
 \cong & \mathrm{Hom}_{G_{n-n_1-n_2}}(\tau, \mathbb D_{\sigma_1\times \sigma_2}(\pi)) .
\end{align*}
where the second, forth and fifth isomorphisms follow from Frobenius reciprocity, the first, third and last ones follow from the adjointness of the functors. We remark that the forth isomorphism also uses taking parabolic inductions in stages. Now the natural isomorphism between the two derivatives follows from the Yoneda lemma.
\end{proof}

\begin{proposition}
Let $\sigma_1, \ldots, \sigma_r \in \mathrm{Irr}^{\square}$ such that $\sigma_1\times \ldots \times \sigma_r$ is still $\square$-irreducible. Suppose $D_{\sigma_1 \times \ldots \times \sigma_r}(\pi) \neq 0$. Then $D_{\sigma_1 \times \ldots \times \sigma_r}(\pi) \cong D_{\sigma_1}\circ \ldots \circ D_{\sigma_r}(\pi)$. 
\end{proposition}

\begin{proof}
Let $\tau=D_{\sigma_1 \times \ldots \times \sigma_r}(\pi)$. Then 
\[  \pi \hookrightarrow \tau \times \sigma_1 \times \ldots \times \sigma_r .
\]
Let $I_{\sigma_1}(\tau)$ be the unique simple submodule of $\tau \times \sigma_1$. The unique submodule must factor through the embedding:
\[   I_{\sigma_1}(\tau) \times \sigma_{2}\times \ldots \times \sigma_r \hookrightarrow \tau \times \sigma_1 \times \ldots \times \sigma_r . 
\]
Then inductively, we have that $D_{\sigma_{2}}\circ \ldots \circ D_{\sigma_r}(\pi) \cong I_{\sigma_1}(\tau)$. This implies that $D_{\sigma_1}\circ D_{\sigma_{2}}\circ \ldots \circ D_{\sigma_r}(\pi)\cong \tau$. 
\end{proof}
\subsection{The segment case}

We now consider a special case of the product functor and we recall the following result shown in \cite{Ch22}. For $\Delta \in \mathrm{Seg}$, let $\mathcal C=\mathcal C_{\Delta}$ be the full subcategory of $\mathrm{Alg}_f(G_n)$ whose objects $\pi$ satisfy the property that for any simple composition factor $\tau$ in $\pi$ and any $\sigma \in \mathrm{csupp}(\tau)$, $\sigma \in \Delta$ (c.f. \cite[Section 9.1]{Ch22}). 

Let $k=l_{abs}(\Delta)$. The product functor 
\[ \times_{\Delta, \mathcal C}: \mathcal C_{\Delta} \rightarrow \mathrm{Alg}(G_{n+k}) \]
as
\[  \times_{\Delta, \mathcal C}(\pi)= \langle \Delta \rangle \times \pi .
\]

\begin{lemma} \label{lem fully faithful segment case}
Let $\Delta\in \mathrm{Seg}$. Then $\times_{\Delta, \mathcal C}$ is a fully-faithful functor.
\end{lemma}
\begin{proof}
This is a special case of \cite[Theorem 9.1]{Ch22}.
\end{proof}



\begin{corollary} \label{cor big derivative multisegment}
Let $\Delta \in \mathrm{Seg}$. Let $\mathfrak m$ be a multisegment with all segments equal to $\Delta$. Let $\mathfrak n$ be a submultisegment of $\mathfrak m$. Then $\mathbb D_{\langle \mathfrak n \rangle}(\langle \mathfrak m \rangle)=\langle \mathfrak m-\mathfrak n \rangle$.
\end{corollary}

\begin{proof}
This follows from Proposition \ref{prop composition of derivatives}, \cite[Corollary 9.2]{Ch22} and Lemma \ref{lem fully faithful segment case}. 
\end{proof}

One may also compare the above two statements with Lemma \ref{lem big derivative st} and Remark \ref{rmk derivative with opposite} below.

\section{Indecomposability under product functor: isomorphic cases }

\subsection{Indecomposibility}

The following result is well-known (see \cite[Proposition 2.3]{Ta90}), but we shall use some similar computations in the proof of Theorem \ref{thm indecomposable length 2} and so we sketch some main steps in the following proof.

\begin{lemma} \label{lem irreducible multisegment sum}
Let $\mathfrak m_1, \mathfrak m_2 \in \mathrm{Mult}$. Suppose $\langle \mathfrak m_1 \rangle \times \langle \mathfrak m_2 \rangle$ is irreducible. Then 
\[ \langle \mathfrak m_1 \rangle \times \langle \mathfrak m_2 \rangle \cong \langle \mathfrak m_1+\mathfrak m_2 \rangle .\]
\end{lemma}

\begin{proof}
 Let $n=l_{abs}(\mathfrak m_1)+l_{abs}(\mathfrak m_2)$.  It suffices to show that 
\begin{eqnarray} \label{eqn hom isomorphic C product}
 \mathrm{Hom}_{G_n}(\langle \mathfrak m_1 \rangle \times \langle \mathfrak m_2 \rangle, \zeta(\mathfrak m_1+\mathfrak m_2 \rangle) \cong \mathbb C. 
\end{eqnarray}
Let $\mathfrak n=\mathfrak m_1+\mathfrak m_2$. Let $\Delta_1, \ldots , \Delta_r$ be all the segments such that $\mathfrak n_{\Delta_i}\neq \emptyset$. We shall label the segments such that $\Delta_1 \not\leq_b \ldots \not\leq_b \Delta_r$. 

Let $s_i=l_{abs}((\mathfrak m_1)_{\Delta_i})$ and let $t_i=l_{abs}((\mathfrak m_2)_{\Delta_i})$ for $i=1, \ldots, r$. Let $u_i=s_i+t_i$. Let 
\[  G'= G_{u_1}\times \ldots \times G_{u_r} ,
\]
and let 
\[  G''=(G_{s_1}\times G_{t_1})\times \ldots \times (G_{s_r}\times G_{t_r}) .
\]

Let $n_1=l_{abs}(\mathfrak m_1)$ and let $n_2=l_{abs}(\mathfrak m_2)$. Note that 
\[ \zeta(\mathfrak m_1+\mathfrak m_2) =\langle \mathfrak n_{\Delta_1} \rangle \times \ldots \times \langle \mathfrak n_{\Delta_r} \rangle .\]
We now apply the Frobenius reciprocity:
\[  \mathrm{Hom}_{G_{u_1}\times G_{n-u_1}}(\langle \mathfrak m_1 \rangle \times \langle \mathfrak m_2 \rangle, \zeta(\mathfrak n_{\Delta_1}) \times \ldots \times \zeta(\mathfrak n_{\Delta_r})) .\]
Then, by Lemma \ref{lem contributing factor of geometric lemma}, a possible layer that can contribute to the non-zero Hom is 
\[ (*)\quad  \mathrm{Ind}^{G_{u_1}\times G_{n-u_1}}_{P} (\langle \mathfrak m_1 \rangle_{N_{n_1-s_1}} \boxtimes \langle \mathfrak m_2 \rangle_{N_{n_2-t_1}} \rangle)^{\phi} ,
\]
where 
\begin{itemize}
\item the superscript $\phi$ is to twist the $G_{s_1}\times G_{n_1-s_1}\times G_{t_1}\times G_{n_2-t_1}$-action to $G_{s_1}\times G_{t_1} \times G_{n_1-s_1}\times G_{n_2-t_1}$ in an obvious way;
\item $P$ is the standard parabolic subgroup containing $G_{s_1}\times G_{t_1} \times G_{n_1-s_1}\times G_{n_2-t_1}$.
\end{itemize}
 Indeed, this is the only possible layer by Lemma \ref{lem contributing factor of geometric lemma}.


In such layer (*), by Lemma \ref{lem contributing factor of geometric lemma} again, the only direct summand that can (possibly) contribute the Hom via Frobenius reciprocity is 
\[   \langle (\mathfrak m_1)_{\Delta_1}\rangle \times \langle  (\mathfrak m_2)_{\Delta_1} \rangle \boxtimes \langle \mathfrak m_1-(\mathfrak m_1)_{\Delta_1} \rangle \times \langle \mathfrak m_2-(\mathfrak m_2)_{\Delta_1} \rangle .
\]
Now we inductively have that $\mathrm{Hom}_{G_{n-s_1-t_1}}(\langle \mathfrak m_1-(\mathfrak m_1)_{\Delta_1}\rangle \times \langle \mathfrak m_2-(\mathfrak m_2)_{\Delta_1} \rangle, \zeta(\mathfrak n-\mathfrak n_{\Delta_1})) \cong \mathbb C$ and so K\"unneth's formula gives (\ref{eqn hom isomorphic C product}).
\end{proof}

The idea of proving the following theorem is that one first enlarges to some modules close to standard modules. In particular, one uses Lemma \ref{lem embedding indecomposable to std} for a module of length $2$. Then one applies Frobenius reciprocity and does some analysis as in the proof of Lemma \ref{lem irreducible multisegment sum}. 

\begin{theorem} \label{thm indecomposable length 2}
Let $\pi_1, \pi_2 \in \mathrm{Irr}$. Suppose $\pi_1 \times \pi_2$ is irreducible. Let $\pi$ be a representation of length $2$ with the two simple composition factors isomorphic to $\pi_2$. Then $\pi$ is indecomposable if and only if $\pi_1 \times \pi$ is indecomposable.
\end{theorem}

\begin{proof}
Let $\mathfrak m_1$ and $\mathfrak m_2$ be multisegments such that 
\[  \pi_1 \cong \langle \mathfrak m_1 \rangle, \quad \pi_2 \cong \langle \mathfrak m_2 \rangle .
\]
Since $\pi_1 \times \pi_2$ is irreducible, $\pi_1 \times \pi_2 \cong \langle \mathfrak m_1 +\mathfrak m_2 \rangle$ by Lemma \ref{lem irreducible multisegment sum}. 

Note that the if direction is clear. We now prove the only if direction. By Lemma \ref{lem embedding indecomposable to std} and taking the dual, there exists $\pi'' \in \mathrm{Alg}_f$ such that 
\begin{itemize}
\item $\pi''$ admits a short exact sequence 
\[   0 \rightarrow \widetilde{\zeta}(\mathfrak m_2) \rightarrow \pi'' \rightarrow \widetilde{\zeta}(\mathfrak m_2) \rightarrow 0
\]
and;
\item $\pi'' \twoheadrightarrow \pi$.
\end{itemize}

Let $n=l_{abs}(\mathfrak m_1)+l_{abs}(\mathfrak m_2)$. We see that $\pi_1 \times \pi''$ is indecomposable if and only if $\mathrm{Hom}_{G_n}(\pi_1 \times \pi'', \langle \mathfrak m_1+\mathfrak m_2 \rangle) \cong \mathbb C$. To prove the latter one, it suffices to show that 
\[  \mathrm{Hom}_{G_n}(\pi_1 \times \pi'', \zeta(\mathfrak m_1 +\mathfrak m_2) ) \cong \mathbb C .
\]

Now we apply some similar strategy as the proof of Lemma \ref{lem irreducible multisegment sum}. Let $\mathfrak n=\mathfrak m_1+\mathfrak m_2$. Let $\Delta_1, \ldots, \Delta_r$ be all the distinct segments such that $\mathfrak n_{\Delta_i}\neq \emptyset$, and $\Delta_1 \not\leq_b \Delta_2 \not\leq_b \ldots \not\leq_b \Delta_r$. Let $s_i=l_{abs}((\mathfrak m_1)_{\Delta_i})$ and $t_i=l_{abs}((\mathfrak m_2)_{\Delta_i})$. Let $u_i=s_i+t_i$. 

Now we write 
\[ \zeta(\mathfrak n) = \zeta(\mathfrak n_{\Delta_1})\times \ldots \times \zeta(\mathfrak n_{\Delta_r}) .
\]
Let $G'=G(\mathfrak n)$. Let $P$ be the standard parabolic subgroup in $G_n$ containing $G'$ with the Levi decomposition $P=G'N$. Recall that $[ \mathfrak n ]=\langle \mathfrak n_{\Delta_1} \rangle \boxtimes \ldots \boxtimes \langle \mathfrak n_{\Delta_r} \rangle$ and so $\zeta(\mathfrak n)=\mathrm{Ind}_P^{G_n}[\mathfrak n]$. Now Frobenius reciprocity gives that:
\[ (*)\quad \mathrm{Hom}_{G_n}(\pi_1 \times \pi'', \zeta(\mathfrak n)) \cong \mathrm{Hom}_{G'}( (\pi_1\times \pi'')_N, [ \mathfrak n ]) .
\]
The analysis in the proof of Lemma \ref{lem irreducible multisegment sum} gives that the only possible layer contributing a non-zero Hom in (*) (via the geometric lemma on $(\pi_1\times \pi'')_N$) is of the form:
\[  \mathrm{Ind}_{\widetilde{P}}^{G'} ((\pi_1)_{N'} \boxtimes (\pi'')_{N''})^{\phi},
\]
where 
\begin{itemize}
\item $N'$ is the unipotent radical corresponding to the partition $(s_1, \ldots, s_r)$ and $N''$ is the unipotent radical corresponding to the partition $(t_1, \ldots, t_r)$;
\item the superscript $\phi$ is a twist sending $G_{s_1}\times \ldots \times G_{s_r} \times G_{t_1}\times \ldots \times G_{t_r}$ to $G_{s_1}\times G_{t_1}\times \ldots \times G_{s_r} \times G_{t_r}$ (by permutating the factors in an obvious way);
\item $\widetilde{P}$ is the standard parabolic subgroup in $G'$ containing $G_{s_1}\times G_{t_1}\times \ldots \times G_{s_r} \times G_{t_r}$. 
\end{itemize}

Thus, we have that:
\[  \mathrm{Hom}_{G_n}(\pi_1 \times \pi'', \zeta(\mathfrak n))  \cong \mathrm{Hom}_{G'}(\mathrm{Ind}_{\widetilde{P}}^{G'} ((\pi_1)_{N'} \boxtimes (\pi'')_{N''})^{\phi}, [ \mathfrak n ]) .
\]
Indeed, as in Lemma \ref{lem irreducible multisegment sum}, which uses Lemma \ref{lem contributing factor of geometric lemma} inductively, the only component in $(\mathrm{Ind}_{\widetilde{P}}^{G'} ((\pi_1)_{N'} \boxtimes (\pi'')_{N''})^{\phi}$ that can contribute to the nonzero Hom is:
\[  \mathrm{Ind}_{\widetilde{P}}^{G'}([ \mathfrak m_1 ] \boxtimes \tau)^{\phi} ,
\] 
where $\tau$ is the direct summand in $(\pi)_{N''}$ whose simple composition factors have the same cuspidal support as $[ \mathfrak m_2 ]$. Attributing to the multiple use of Lemma \ref{lem variation of cuspidal supprot}, we conclude that $\tau$ has length $2$ and both composition factors of $\tau$ are isomorphic to $[\mathfrak m_2]$. Thus we further have
\[(**) \quad \mathrm{Hom}_{G_n}(\pi_1 \times \pi'', \zeta(\mathfrak n))  \cong \mathrm{Hom}_{G'}(\mathrm{Ind}_{\widetilde{P}}^{G'} ([ \mathfrak m_1] \boxtimes \tau)^{\phi}, [ \mathfrak n ]) .
\]

As the functors described in the proof of Corollary \ref{cor equivalence of categories}, $\tau$ satisfies $\mathrm{Ind}^{G_{n_2}}_{P^*} \tau  \cong \pi''$, where $P^*$ is the standard parabolic subgroup containing $G_{t_1}\times \ldots \times G_{t_r}$, and Corollary \ref{cor equivalence of categories} implies that $\tau$ is indecomposable and of length $2$ with both factors isomorphic to $[\mathfrak m_2]$. In particular, $\tau$ has a unique simple quotient. 

Now we return to compute the latter Hom of (**). Let 
\[  G'' = G_{s_1}\times G_{t_1}\times \ldots \times G_{s_r}\times G_{t_r} .
\]
In such case, applying the second adjointness, such Hom is equal to 
\[  \mathrm{Hom}_{G''}(([ \mathfrak m_1 ]\boxtimes \tau))^{\phi}, [ \mathfrak n ]_{N^-}) ,
\]
where $N^-$ is the unipotent radical in $P_{s_1,t_1}^t\times \ldots \times P_{s_r, t_r}^t \subset G_{u_1}\times \ldots \times G_{u_r}$. By using Hom-tensoring adjointness, the previous Hom turns to:
\[ (***) \quad \mathrm{Hom}_{G_{t_1}\times \ldots \times G_{t_r}}(\tau, \mathrm{Hom}_{G_{s_1}\times \ldots \times G_{s_r}}([ \mathfrak m_1 ], [ \mathfrak n ])) ,
\]
where we regard $[ \mathfrak n ]$ as a $G_{s_1}\times \ldots \times G_{s_r}$-representation via the embedding: 
\[  (g_1, \ldots g_r) \mapsto  (\begin{pmatrix} g_1 & \\ & I_{t_1} \end{pmatrix}, \ldots, \begin{pmatrix} g_r & \\ & I_{t_r} \end{pmatrix}).
\]
Let $\sigma_i=\langle (\mathfrak m_1)_{\Delta_i} \rangle$. Finally, using K\"unneth formula on (***) and combining with (*), we have that: 
\[ \mathrm{Hom}_{G_n}(\pi_1 \times \pi'', \zeta(\mathfrak n))  \cong \mathrm{Hom}_{G_{t_1}\times \ldots \times G_{t_r}}( \tau, \mathbb D_{\sigma_1}(\langle \mathfrak n_{\Delta_1} \rangle) \boxtimes \ldots \boxtimes \mathbb D_{\sigma_r}(\langle \mathfrak n_{\Delta_r} \rangle))
\]
and so, by Corollary \ref{cor big derivative multisegment}, 
\[ \mathrm{Hom}_{G_n}(\pi_1 \times \pi'', \zeta(\mathfrak n))  \cong  \mathrm{Hom}_{G_{t_1}\times \ldots \times G_{t_r}}(\tau, \langle \mathfrak n'_1 \rangle \boxtimes \ldots \boxtimes \langle \mathfrak n'_r \rangle),
\]
where $\mathfrak n_i'=(\mathfrak n)_{\Delta_i}-(\mathfrak m_1)_{\Delta_i}$. Now, as discussed above $\tau$ has only unique simple quotient and we so have that the Hom space has dimension $1$, as desired. Thus, we have $\mathrm{Hom}_{G_n}(\pi_1 \times \pi'', \zeta(\mathfrak m_1 +\mathfrak m_2) ) \cong \mathbb C$ and so $\mathrm{Hom}_{G_n}(\pi_1\times \pi'', \langle \mathfrak m_1+\mathfrak m_2 \rangle) \cong \mathbb C$.
\end{proof}

\begin{remark} \label{rmk example no preserve self ext}
In general, the parabolic induction does not preserve self-extensions. For example, we consider $\pi =\langle [0] \rangle$. Let $\tau=(\pi \times \pi)_{N_1}$. Then 
\[ \mathrm{Ind}_{P_{1,1}}^{G_2} \tau \cong \pi \times\pi \oplus \pi \times \pi. \] 
\end{remark}

\section{Fully-faithfulness of the product functor}

\subsection{A criteria for proving fully-faithfulness}

 We recall the following criteria for proving fully-faithfulness:

\begin{lemma} \label{lem fully faithful functor criteria} \cite[Lemma A.1]{Ch22}
Let $\mathcal A$ and $\mathcal B$ be abelian Schurian $k$-categories, where $k$ is a field. Let $F: \mathcal A \rightarrow \mathcal B$ be an additive exact functor. Suppose the following holds:
\begin{enumerate}
\item any object of $\mathcal A$ has finite length;
\item for any simple objects $X$ and $Y$ in $\mathcal A$, the induced map of $F$ from $\mathrm{Ext}^1_{\mathcal A}(X, Y)$ to $\mathrm{Ext}^1_{\mathcal B}(F(X), F(Y))$ is an injection;
\item $F(X)$ is a simple object in $\mathcal B$ if $X$ is a simple object in $\mathcal A$;
\item for any simple objects $X$ and $Y$ in $\mathcal A$, $X \cong Y$ if and only if $F(X) \cong F(Y)$. 
\end{enumerate}
Then $F$ is a fully-faithful functor.
\end{lemma}

The original statement of \cite[Lemma A.1]{Ch22} is stated in a slightly different way, but the proof still applies.

We remark that elements in $\mathrm{Ext}^1_{\mathcal A}(X, Y)$ can be interpreted as short exact sequences in Yoneda extensions \cite{Ma63}, and the addition corresponds to the Baer sum. Then the exact functor $F: \mathcal A\rightarrow \mathcal B$ sends a short exact sequence to a short exact sequence and so induces a map from $\mathrm{Ext}^1_{\mathcal A}(X,Y)$ to $\mathrm{Ext}^1_{\mathcal B}(F(X), F(Y))$ above.

If we consider the full subcategory $\mathcal B'$ of $\mathcal B$ containing all $F(X)$ for objects $X$ in $\mathcal A$, then $F$ defines an equivalence of categories from $\mathcal A$ to $\mathcal B'$ \cite[Lemma 4.2.19]{St}. We remark that $\mathcal B'$ may not be Serre in $\mathcal B$.

\subsection{Product functor}

Recall that $\mathrm{Alg}_{\pi}(G_n)$ is defined in Section \ref{ss indecomposable repn}.

\begin{theorem} \label{thm fully faithful functor}
Let $\pi \in \mathrm{Irr}$. Let $k=\mathrm{deg}(\pi)$. Let $\mathcal C=\mathrm{Alg}_{\pi}(G_n)$. The functor $\times_{\pi, \mathcal C}$ defined in Section \ref{ss product functor} is fully-faithful i.e.
\[  \mathrm{Hom}_{G_{n+k}}(\times_{\pi, \mathcal C}(\tau_1), \times_{\pi, \mathcal C}(\tau_2)) \cong \mathrm{Hom}_{G_n}(\tau_1, \tau_2) 
\]
for any $\tau_1, \tau_2 \in \mathrm{Alg}_{\pi}(G_n)$. 
\end{theorem}


\begin{proof}
It suffices to check the conditions in Lemma \ref{lem fully faithful functor criteria}. For (1), it follows from the definitions. For (3), it follows from Proposition \ref{prop segment indecomp embed}. For (4), it follows from Lemma \ref{lem irreducible multisegment sum}. For (2), it follows from Proposition \ref{prop indecomp non-isomorphic} and Theorem \ref{thm indecomposable length 2}.  
\end{proof}

\begin{remark} \label{rmk big derivative exmaple}
Let $\pi \in \mathrm{Irr}$. Let $\lambda$ be an indecomposable representation such that for any $\pi' \in \mathrm{JH}(\lambda)$, $\pi \times \pi'$ is irreducible. In general, it is not necessary that $\pi \times \lambda$ is indecomposable. For example, take $\pi=\langle [0] \rangle$ and let $\lambda=\langle [1] \rangle \times \langle [0] \rangle$. Then $\pi \times \lambda$ is a direct sum of two irreducible representations.
\end{remark}

\begin{corollary} \label{cor fully fiathful left version}
The functor $\times^{\pi, \mathcal C}: \mathcal C \rightarrow \mathrm{Alg}(G_{n+k})$ determined by $\tau \mapsto \tau \times \pi$ is fully-faithful.
\end{corollary}

\begin{proof}
When $D=F$, it follows from Theorem \ref{thm fully faithful functor} and Proposition \ref{prop fully faithful segment}. In general, using the Zelevinsky type classification, one can prove in a similar way to Theorem \ref{thm fully faithful functor}.
\end{proof}

\subsection{Dual formulation}

\begin{theorem}
Let $\pi \in \mathrm{Irr}$. Let 
\[  \mathcal N_{\pi}=\left\{ \mathfrak m \in \mathrm{Mult}: \mathrm{St}(\Delta)\times \pi \mbox{ is irreducible } \forall \Delta \in \mathfrak m \right\} . \]
Let $\mathcal C':=\mathrm{Alg}_{\pi}'(G_n)$ be the full subcategory of $\mathrm{Alg}_f(G_n)$ whose objects have all simple composition factors isomorphic to $\mathrm{St}(\mathfrak m)$ for some $\mathfrak m \in \mathcal N_{\pi}$. Then the product functors $\times_{\pi, \mathcal C'}$ and $\times^{\pi, \mathcal C'}$ are fully-faithful. 
\end{theorem}

\begin{proof}
Note that $\mathcal D(\langle \Delta \rangle)=\mathrm{St}(\Delta^{\vee})$ by using the formulation in \cite{SS97}. This follows from Theorem \ref{thm fully faithful functor}, Corollary \ref{cor fully fiathful left version} and Corollary \ref{cor dual statement}.
\end{proof}

\subsection{Self-extensions}

We now study the fully-faithfulness of some $\square$-irreducible representations. One may compare with Lemma \ref{lem fully faithful segment case}. 

\begin{proposition}
Let $\pi \in \mathrm{Irr}(G_l)$. Let $\pi' \in \mathrm{Irr}$ such that $\pi \times \pi'$ is irreducible. Let $\mathrm{Alg}_{\pi', \mathrm{self}}(G_n)$ be the full subcategory of $\mathrm{Alg}_f(G_n)$ whose objects have all simple composition factors isomorphic to $\pi'$. Then the product functor 
\[ \times_{\pi}^s: \mathrm{Alg}_{\pi', \mathrm{self}}(G_n) \rightarrow \mathrm{Alg}(G_{n+l})  \]
given by
\[  \times_{\pi}^s(\tau)=\pi \times \tau
\]
 is fully-faithful.
\end{proposition}
\begin{proof}
Again we check the conditions in Lemma \ref{lem fully faithful functor criteria}. (1), (3) and (4) are automatic. (2) follows from Theorem \ref{thm indecomposable length 2}. 
\end{proof}

\begin{remark}
For $\pi \in \mathrm{Irr}^{\square}$, $\mathbb D_{\pi}(\pi \times \pi)$ is possibly not irreducible. For example, take $\mathfrak m=\left\{ [0,1], [1] \right\}$. Let $\pi=\langle \mathfrak m \rangle$. Via a geometric lemma consideration, one deduces that:
\[   (\langle [0] \rangle \times \langle [1] \rangle \times \langle [1] \rangle \times \langle [0,1] \rangle) \boxtimes \langle [1] \rangle 
\]
is a submodule of $\mathbb D_{[1]}(\pi \times \pi)$ (also see Example \ref{exp big derivative example} below). Then 
\[   ( \langle [0] \rangle \times \langle [1] \rangle \times \langle [1] \rangle) \boxtimes \pi 
\]
is a submodule of $\mathbb D_{[0,1]}\circ \mathbb D_{[1]}(\pi \times \pi) \cong \mathbb D_{\pi}(\pi \times \pi)$ (see Proposition \ref{prop composition of derivatives}). 

One may further ask when $\pi \in \mathrm{Irr}^{\square}$ is prime in the Bernstein-Zelevinsky ring, is it true that $\mathbb D_{\pi}(\pi \times \pi)=\pi$? This holds for when $\pi$ is a Speh representation by using \cite{Ch22}, but it is not clear for the general situation.
\end{remark}

\section{Application on the SI property for big derivatives} \label{s application si derivatives}

\subsection{More notations on derivatives}

Recall that the big derivative is defined in Definition \ref{def big derivatives}. For $\Delta \in \mathrm{Seg}$, set $\mathbb D_{\Delta} = \mathbb D_{\mathrm{St}(\Delta)}$. 

For $\pi \in \mathrm{Irr}$, let $D_{\Delta}(\pi)$ be the unique submodule of $\mathbb D_{\Delta}(\pi)$ if $\mathbb D_{\Delta}(\pi)\neq 0$. Let $D_{\Delta}(\pi)=0$ if $\mathbb D_{\Delta}(\pi)=0$. 

For a generic multisegment $\mathfrak m$, we similarly set $\mathbb D_{\mathfrak m}=\mathbb D_{\mathrm{St}(\mathfrak m)}$. For $\pi \in \mathrm{Irr}$, we similarly define $D_{\mathfrak m}(\pi)$ as the unique submodule of $\mathbb D_{\mathfrak m}(\pi)$ if $\mathbb D_{\mathfrak m}(\pi)\neq 0$ and define $D_{\mathfrak m}(\pi)=0$ otherwise. The uniqueness of the simple submodule in $D_{\Delta}(\pi)$ and $\mathbb D_{\mathfrak m}(\pi)$ follows from \cite{LM16} and \cite{KKKO15}. 

Set $i=l_{abs}(\mathfrak m)$. With a slight reformulation from above, we also have that: 
\[  \pi_{N_i}  \twoheadrightarrow  D_{\mathfrak m}(\pi)\boxtimes \mathrm{St}(\mathfrak m) ,
\]
or equivalently, by \cite[Corollary 2.4]{LM19},
\[  D_{\mathfrak m}(\pi)\boxtimes \mathrm{St}(\mathfrak m) \hookrightarrow \pi_{N_i},
\]
or equivalently, by Frobenius reciprocity,
\[   \pi \hookrightarrow D_{\mathfrak m}(\pi)\times \mathrm{St}(\mathfrak m) .
\]


\subsection{$\eta$-invariants and $\Delta$-reduced representations} \label{ss eta invariant}

We shall first discuss more results on derivatives.

Define $\epsilon_{\Delta}(\pi)$ to be the largest non-negative integer $k$ such that 
\[   D_{\Delta}^k(\pi) \neq 0 .
\]
Define 
\[ \eta_{\Delta}(\pi) =(\epsilon_{[a,b]_{\rho}}(\pi), \epsilon_{[a-1,b]_{\rho}}(\pi), \ldots, \epsilon_{[b,b]_{\rho}}(\pi)) .\]
Using similar terminologies in \cite[Section 7]{LM22}, a segment $[c,d]_{\rho}$ is said to be {\it $[a,b]_{\rho}$-saturated} if $d=b$ and $a\leq c$. Define $\mathfrak{mx}(\pi, \Delta)$ to be the multisegment that contains exactly the $\Delta$-saturared segments $\Delta'$ with the multiplicity $\epsilon_{\Delta'}(\pi)$. We shall call $\pi$ to be {\it $\Delta$-reduced} if $\mathfrak{mx}(\pi, \Delta)=\emptyset$. 


We give two useful properties related to the $\eta$-invariant, which will also be used in the appendix. Those properties are also useful in the study of the Bernstein-Zelevinsky derivatives \cite{Ch22+, Ch22+c}:

\begin{proposition} \label{prop mx multisegment direct summand} (c.f. \cite[Proposition 7.3]{LM22})
Let $\pi \in \mathrm{Irr}$ and let $\Delta \in \mathrm{Seg}$. Let $\mathfrak p=\mathfrak{mx}(\pi, \Delta)$. Let $i=l_{abs}(\mathfrak p)$. Then $D_{\mathfrak p}(\pi)\boxtimes \mathrm{St}(\mathfrak p)$ is a direct summand in $\pi_{N_i}$. 
\end{proposition}

\begin{proof}
Let $\tau =D_{\mathfrak p}(\pi)$. Then $\tau$ is $\Delta$-reduced. We also have the embedding:
\[   \pi \hookrightarrow \tau \times \mathrm{St}(\mathfrak p) .
\]
Then we apply the Jacquet functor ${}_{N_i}$ on $\tau \times \mathrm{St}(\mathfrak p)$. We first mention two important ingredients. The first one is to use the Jacquet functor computations for (\ref{eqn jacquet on steinberg repn}). The second one is that the $\Delta$-reduced property on $\tau$ implies that if a simple composition factor in $\tau_{N_j}$ (for some $j$) takes the form $\omega_1\boxtimes \omega_2$ and satisfies that $\mathrm{csupp}(\omega_2) \subset \cup_{\Delta' \in \mathfrak p}\Delta'$ (as a multiset), then $b(\Delta) \notin \mathrm{csupp}(\omega_2)$. 

Now, we have to see which layer in the geometric lemma can contribute to the same cuspidal support as $\tau \boxtimes \mathrm{St}(\mathfrak p)$. But, the second point implies that, all $b(\Delta)$ must come from a factor from $\mathrm{St}(\mathfrak p)$. However, the first point will then force that the layer in in $\tau \times \mathrm{St}(\mathfrak p))_{N_i}$ must come from the layer of the form $\tau \boxtimes \mathrm{St}(\mathfrak p)$.  Thus $\tau\boxtimes \mathrm{St}(\mathfrak p)$ is a direct summand in $(\tau \times \mathrm{St}(\mathfrak p))_{N_i}$ and so is a direct summand in $\pi_{N_i}$.
\end{proof}

 We shall use it to do a reduction later. When $\Delta$ is a singleton, Proposition \ref{prop mx multisegment direct summand} is also shown by Jantzen \cite{Ja07} and M\'inguez \cite{Mi09} (also see \cite{GV01}).

When $\mathfrak m \in \mathrm{Mult}$ is generic, $\mathrm{St}(\mathfrak m)=\lambda(\mathfrak m)$ and $\mathrm{St}(\mathfrak m)$ is generic when $D=F$ \cite{Ze80}. For generic $\mathfrak m \in \mathrm{Mult}$ and $\pi \in \mathrm{Irr}$, denote by $I_{\mathfrak m}(\pi)$ the unique simple submodule of $\pi \times \mathrm{St}(\mathfrak m)$. (Here the uniqueness follows from \cite{KKKO15} and \cite{LM19} since $\mathrm{St}(\mathfrak m) \in \mathrm{Irr}^{\square}$.)

A multisegment $\mathfrak m$ is said to be {\it $\Delta$-saturated} if all the segments in $\mathfrak m$ are $\Delta$-saturated. We denote by $\mathrm{Mult}_{\Delta-sat}$ the set of all $\Delta$-saturated multisegments.


\begin{proposition} \label{prop composition factor mx multiseg} (c.f. \cite[Proposition 7.3]{LM22})
Let $\Delta \in \mathrm{Seg}$. Let $\tau \in \mathrm{Irr}$. Suppose $\tau$ is $\Delta$-reduced. Let $\mathfrak p \in \mathrm{Mult}_{\Delta-sat}$. Then,
\begin{itemize}
\item[(1)] $I_{\mathfrak p}(\tau)$ appears with multiplicity one in $\tau \times \mathrm{St}(\mathfrak p)$;
\item[(2)] for $\pi'$ in $\mathrm{JH}(\tau \times \mathrm{St}(\mathfrak p))$ with $\pi'\not\cong I_{\mathfrak p}(\tau)$, $\mathfrak{mx}(\pi', \Delta) \neq \mathfrak p$.
\end{itemize}
\end{proposition}

\begin{proof}
Note that $\tau \cong D_{\mathfrak p}\circ I_{\mathfrak p}(\tau)$ by definitions. Then, we again have that:
\[   I_{\mathfrak p}(\tau) \hookrightarrow \tau \times \mathrm{St}(\mathfrak p) .
\]
In Proposition \ref{prop mx multisegment direct summand}, we showed that $\tau \boxtimes \mathrm{St}(\mathfrak p)$ appears with multiplicity one in $(\tau \times \mathrm{St}(\mathfrak p))_{N_l}$, where $l=l_{abs}(\mathfrak p)$. This implies (1). Moreover, the proof of Proposition \ref{prop mx multisegment direct summand} also shows that no other composition factor in $(\tau \times \mathrm{St}(\mathfrak p))_{N_l}$ takes the form $\tau'\boxtimes \mathrm{St}(\mathfrak p)$. This implies (2).
\end{proof}

\subsection{SI property on the segment case} \label{ss si property of segment case}

We introduce one more notions for convenience. For a cuspidal representation $\rho$, a multisegment $\mathfrak p$ is said to be {\it strongly $\rho$-saturated} if $b(\Delta)\cong \rho$ for any segment $\Delta$ in $\mathfrak p$.

\begin{lemma} \label{lem big derivative st} 
Fix a cuspidal representation $\rho$. Let $\mathfrak p$ be a strongly $\rho$-saturated multisegment. Let $\Delta \in \mathfrak p$. Then 
\[  \mathbb D_{\Delta}(\mathrm{St}(\mathfrak p)) =D_{\Delta}(\mathrm{St}(\mathfrak p))=\mathrm{St}(\mathfrak p-\Delta). 
\] 
\end{lemma}

\begin{proof}
Let $\pi=\mathrm{St}(\Delta)$. By \cite{Ze80} and \cite{Ta90}, $\mathrm{St}(\mathfrak p-\Delta) \in \mathcal M_{\pi}$. Note that $\mathrm{St}(\mathfrak p)=\mathrm{St}(\mathfrak p-\Delta)\times \mathrm{St}(\Delta)$. 

Let $n=l_{abs}(\mathfrak p-\Delta)$ and let $k=l_{abs}(\Delta)$. We consider the full subcategory $\mathcal A'$ of $\mathrm{Alg}(G_n)$ whose objects have all composition factors isomorphic to $\mathrm{St}(\mathfrak p-\Delta)$, and let $\mathcal B'$ be the full subcategory of $\mathrm{Alg}(G_{n+k})$ whose objects have all composition factors isomorphic to $\mathrm{St}(\mathfrak p)$. By Theorem \ref{thm fully faithful functor}, $\times_{\pi, \mathcal A'}: \mathcal A' \rightarrow \mathcal B'$ is a fully-faithful functor. Moreover, $\mathbb D_{\Delta}: \mathcal B' \rightarrow \mathcal A'$ is well-defined and is right-adjoint to $\times_{\pi, \mathcal A'}$. Thus, 
\[ \mathbb D_{\Delta}(\mathrm{St}(\mathfrak p))=\mathbb D_{\Delta}(\times_{\mathrm{St}(\Delta)}(\mathrm{St}(\mathfrak p-\Delta)))=\mathrm{St}(\mathfrak p-\Delta) \]
is irreducible, by \cite[Lemma 4.24.3]{St}. 
\end{proof}

\begin{remark} \label{rmk derivative with opposite}
The statement of Lemma \ref{lem big derivative st} does not hold in general if we simply replace $\mathbb D_{\Delta}$ by $\mathbb D_{\Delta}'$. 
\end{remark}

For $\sigma \in \mathrm{Irr}^{\square}$ and $\pi \in \mathrm{Irr}$, by \cite{KKKO15}, there exists at most one $\tau \in \mathrm{Irr}$ such that 
\[  \tau \boxtimes \sigma \hookrightarrow \pi_{N_{\mathrm{deg}(\sigma)}} .
\]
If such $\tau$ exists, denote such $\tau$ by $D_{\sigma}(\pi)$. Otherwise, set $D_{\sigma}(\pi)=0$. 


\begin{proposition} \label{prop multiplicity one segment case}
Fix a cuspidal representation $\rho$. Let $\mathfrak p$ be a strongly $\rho$-saturated multisegment. Let $\sigma=\mathrm{St}(\mathfrak p)$ and let $\pi \in \mathrm{Irr}$. Suppose $D_{\sigma}(\pi)\neq 0$. Then $\mathbb D_{\sigma}(\pi)$ is SI. 
\end{proposition}

\begin{proof}
We write $\sigma=\mathrm{St}(\Delta)$ for some segment $\Delta$. Let $\mathfrak p=\mathfrak{mx}(\pi, \Delta)$. Then we have that
\[  \pi\hookrightarrow  D_{\mathfrak p}(\pi) \times \mathrm{St}(\mathfrak p) .
\]
Now, one has:
\[   \mathbb D_{\Delta}(\pi) \hookrightarrow \mathbb D_{\Delta}(D_{\mathfrak p}(\pi) \times \mathrm{St}(\mathfrak p)).
\]
From Definition \ref{def big derivatives}, one has to compute a Jacquet module on $D_{\mathfrak p}(\pi) \times \mathrm{St}(\mathfrak p)$ and the structure again can be computed from the geometric lemma. With a standard cuspidal support argument before, one boils down to have:
\[  \mathbb D_{\Delta}(\pi) \hookrightarrow D_{\mathfrak p}(\pi) \times \mathbb D_{\Delta}(\mathrm{St}(\mathfrak p)). \]

By Lemma \ref{lem big derivative st}, we have that
\[  \mathbb D_{\Delta}(\pi) \hookrightarrow D_{\mathfrak p}(\pi) \times \mathrm{St}(\mathfrak p-\Delta) . 
\]
Now $D_{\Delta}(\pi)$ is the unique submodule of $D_{\mathfrak p}(\pi)\times \mathrm{St}(\mathfrak p-\Delta)$ and appears with multiplicity one in $D_{\mathfrak p}(\pi) \times \mathrm{St}(\mathfrak p-\Delta)$ by \cite{LM16} or \cite{KKKO15}. Hence $D_{\Delta}(\pi)$ also appears with multiplicity one in $\mathbb D_{\Delta}(\pi)$. 
\end{proof}

\begin{example} \label{exp big derivative example}
For a segment $\Delta$ and $\pi \in \mathrm{Irr}$, $\mathbb D_{\Delta}(\pi)$ is not irreducible in general. For example, let $\mathfrak m=\left\{ [1], [0,1] \right\}$. Then $\mathbb D_{[1]}(\langle \mathfrak m\rangle)$ has length two. 

One may further consider the indecomposable component $\tau$ in $\langle \mathfrak m \rangle_{N_1}$ which contains $\langle [0,1] \rangle \boxtimes \langle [1] \rangle$ as the submodule. It is shown by (some variants of) \cite[Corollary 2.9]{Ch21} (also see \cite{Ch22+}) that $\tau$ is the direct summand with all the simple composition factors in $\langle \mathfrak m \rangle_{N_1}$ with the same cuspidal support as $\langle [0,1] \rangle \boxtimes \langle [1] \rangle$. Thus we have the following relation:
\[  D_{[1]}(\langle \mathfrak m \rangle) \boxtimes \langle [1] \rangle \subsetneq  \mathbb D_{[1]}(\langle \mathfrak m \rangle) \boxtimes \langle [1] \rangle \subsetneq \tau.
\]
\end{example}

\subsection{An application}

We give an application on studying how to embed some Jacquet modules into some layers arising from the geometric lemma. The study of how to do such embedding will be used in \cite{Ch22+b, Ch22+c} for studying commutations of some derivatives and integrals.


\begin{proposition} \label{prop mxmultisegment and big derivative isomorphic}
Let $\pi=\mathrm{St}(\mathfrak n)$ for some generic $\mathfrak n \in \mathrm{Mult}_n$. Let $\Delta \in \mathrm{Seg}$ such that $D_{\Delta}(\pi)\neq 0$.  Let $\tau$ be the unique indecomposable component with the unique submodule $D_{\Delta}(\pi)\boxtimes \mathrm{St}(\Delta)$ in $\pi_{N_{l_{abs}(\Delta)}}$. Then $\mathfrak{mx}(\pi, \Delta)$ contains only one segment if and only if 
\[   D_{\Delta}(\pi)\boxtimes \mathrm{St}(\Delta) \cong \mathbb D_{\Delta}(\pi)\boxtimes \mathrm{St}(\Delta) \cong \tau .
\]
\end{proposition}

\begin{proof}
Let $\mathfrak n$ be the multisegment such that $\pi \cong \mathrm{St}(\mathfrak n)$. Let $\mathfrak m=\mathfrak{mx}(\pi, \Delta)$. Let $i=l_{abs}(\Delta)$. If $\mathfrak m$ contains only one segment, then $\mathfrak m=\left\{ \Delta \right\}$. Then $D_{\Delta}(\pi)\boxtimes \mathrm{St}(\Delta)$ is a direct summand in $\pi_{N_{i}}$ (see Proposition \ref{prop mx multisegment direct summand}). Thus we have $D_{\Delta}(\pi)=\mathbb D_{\Delta}(\pi)$ and $\tau \cong D_{\Delta}(\pi)\boxtimes \mathrm{St}(\Delta)$. 

We now prove the converse direction. Since we are dealing with the generic case, we have a simple description on $\mathfrak{mx}$ as:
\[ \mathfrak{mx}(\pi, \Delta)=\left\{ [a(\Delta'), b(\Delta)] : \Delta' \in \mathfrak n, \quad a(\Delta') \leq b(\Delta) \right\}
\]
Thus a geometric lemma shows that $D_{\Delta}(\pi)\boxtimes \mathrm{St}(\Delta)$ appears more than one time in $\pi_{N_i}$. 

Let $\omega=D_{\mathfrak m}(\pi)$. 
On the other hand, if $\mathbb D_{\Delta}(\pi)\boxtimes \mathrm{St}(\Delta) \cong \tau$, then 
\[  \tau \cong \mathbb D_{\Delta}(\pi) \boxtimes \mathrm{St}(\Delta) \hookrightarrow \mathbb D_{\Delta}( \omega \times \mathrm{St}(\mathfrak m)) \boxtimes \mathrm{St}(\Delta) .
\]
Again, computing $\mathbb D_{\Delta}( \omega \times \mathrm{St}(\mathfrak m))$ involves a compution of a Jacquet module on $\omega \times \mathrm{St}(\mathfrak m)$, which leads to analyzing on layers in the geometric lemma. Again with a standard comparison on cuspidal support, there is only one layer contributing the submodule $D_{\Delta}(\pi)$, that is of the form $\omega \times \mathbb D_{\Delta}(\mathrm{St}(\mathfrak m))$. Since that layer appears in the toppest one, we must then have that:
\[  \tau \hookrightarrow  (\omega \times \mathbb D_{\Delta}(\mathrm{St}(\mathfrak m))) \boxtimes \mathrm{St}(\Delta) .
\]
By Lemma \ref{lem big derivative st}, we then have:
\begin{align} \label{eqn multiplicity one}
  \tau \hookrightarrow (\omega \times \mathrm{St}(\mathfrak m-\Delta)) \boxtimes \mathrm{St}(\Delta) .
\end{align}
Since $\mathfrak{mx}(\omega, \Delta)=\emptyset$, Proposition \ref{prop composition factor mx multiseg} gives that $D_{\Delta}(\pi)$ appears with multiplicity one in $\omega \times \mathrm{St}(\mathfrak m-\Delta)$ and so as in $\tau$. This contradicts to what we argued before. Hence, we cannot have $\tau \cong \mathbb D_{\Delta}(\pi)\boxtimes \mathrm{St}(\Delta)$. 
\end{proof}

An alternate way to see Proposition \ref{prop mxmultisegment and big derivative isomorphic} is that if $\mathfrak{mx}(\pi, \Delta)$ contains more than one segment, then $\tau$ cannot be written into the form $\omega' \boxtimes \mathrm{St}(\Delta)$ for some $\omega' \in \mathrm{Alg}_f$ since this otherwise will imply $\mathbb D_{\Delta}(\pi)\cong \omega'$ and so $\tau \cong \mathbb D_{\Delta}(\pi)\boxtimes \mathrm{St}(\Delta)$ giving a contradiction. This consequently gives:

\begin{corollary}
Let  $\pi=\mathrm{St}(\mathfrak n)$ for some generic $\mathfrak n \in \mathrm{Mult}$. Let $\Delta$ be a segment such that $D_{\Delta}(\pi)\neq 0$. Let $\omega$ be a representation of finite length such that 
\[ \pi \hookrightarrow \omega \times \mathrm{St}(\Delta) .\]
Let $i=l_{abs}(\Delta)$. Let $p: \pi_{N_i} \twoheadrightarrow \omega \boxtimes \mathrm{St}(\Delta)$ be the projection arising from the geometric lemma (see Section \ref{ss geo lemma}). Let $\iota: D_{\Delta}(\pi)\boxtimes \mathrm{St}(\Delta) \hookrightarrow \pi_{N_i}$ be the unique embedding.  Suppose $\mathfrak{mx}(\pi, \Delta)$ contains at least two segments. Then $p\circ \iota=0$.
\end{corollary}

\begin{proof}
Let $\tau$ be the unique indecomposable module in $\pi_{N_i}$ that contains $D_{\Delta}(\pi)\boxtimes \mathrm{St}(\Delta)$ as submodule. 

We have the following short exact sequence:
\[  0 \rightarrow \kappa   \rightarrow \pi_{N_i} \stackrel{p}{\rightarrow} \omega \boxtimes \mathrm{St}(\Delta)\rightarrow 0 ,
\]
where $\kappa$ is the kernel of the projection $p$. If $\tau \cap \kappa \neq 0$, then $D_{\Delta}(\pi)\boxtimes \mathrm{St}(\Delta)$ must contain the unique submodule $D_{\Delta}(\pi)\boxtimes \mathrm{St}(\Delta)$ by the uniqueness of simple submodule in $\tau$. Thus it suffices to show $\tau \cap \kappa \neq 0$. Suppose not. Then, 
 \[  \tau \hookrightarrow \omega \boxtimes \mathrm{St}(\Delta).
\]
This implies that $\tau \cong \omega' \boxtimes \mathrm{St}(\Delta)$ for some submodule $\omega'$ of $\omega$. Then $\mathbb D_{\Delta}(\pi) \boxtimes \mathrm{St}(\Delta) \cong \tau$. This contradicts Proposition \ref{prop mxmultisegment and big derivative isomorphic}. 
\end{proof}

\section{Appendix: SI property of big derivatives for generic representations}

It is interesting to generalize Proposition \ref{prop multiplicity one segment case} to a larger family of big derivatives. We shall explain how to extend to generic representations in this appendix.

\subsection{Lemma on $\Delta$-reduced representations}

We generalize the idea of $\eta$-invariants in Section \ref{ss eta invariant} to representations of finite lengths. 

\begin{definition}
Let $\pi \in \mathrm{Alg}_f$. Let $\Delta \in \mathrm{Seg}$. We say that $\pi$ is {\it $\Delta$-reduced} if for any $\Delta$-saturated segment $\widetilde{\Delta}$, $\mathbb D_{\widetilde{\Delta}}(\pi)=0$.
\end{definition}

The following lemma is a simple exercise using the Jacquet functors and we shall omit the details.

\begin{lemma} \label{lem big derivative zero mxpt}
Let $\pi \in \mathrm{Alg}_f$. Then $\pi$ is $\Delta$-reduced if and only if $\pi'$ is $\Delta$-reduced for any $\pi'$ in $\mathrm{JH}(\pi)$. 
\end{lemma}

\subsection{Generic case}

We first have the following commutativity result:
\begin{lemma} \label{lem commutativity unlinked seg}
Let $\Delta_1, \Delta_2$ be two unlinked segments. Let $\pi \in \mathrm{Alg}_f$. Then 
\[  \mathbb D_{\Delta_1}\circ \mathbb D_{\Delta_2}(\pi) \cong \mathbb D_{\Delta_2}\circ \mathbb D_{\Delta_1}(\pi).
\]
\end{lemma}

\begin{proof}
Since $\mathrm{St}(\Delta_1)\times \mathrm{St}(\Delta_2) \cong \mathrm{St}(\Delta_2)\times \mathrm{St}(\Delta_1)$ is irreducible by \cite{Ze80, Ta90}, the result follows from Proposition \ref{prop composition of derivatives}.
\end{proof}

\begin{lemma} \label{lem derivative zero for mx zero}
Let $\pi \in \mathrm{Irr}$. Let $\Delta \in \mathrm{Seg}$. Suppose $\pi$ is $\Delta$-reduced. Let $\Delta'$ be a segment such that $a(\Delta) \leq a(\Delta') \leq b(\Delta) \leq b(\Delta')$. Then $\mathbb D_{\Delta'}(\pi)=D_{\Delta'}(\pi)=0$. 
\end{lemma}

\begin{proof}
Suppose $D_{\Delta'}(\pi)\neq 0$. Let $i=l_{abs}(\Delta')$. Then we have an embedding:
\[    D_{\Delta'}(\pi) \boxtimes \mathrm{St}(\Delta') \hookrightarrow \pi_{N_i}.
\]
Let $\Delta''=[a(\Delta'), b(\Delta)]$. Let $n=\mathrm{deg}(\pi)$ and let $j=l_{abs}(\Delta'')$. Write $\Delta=[a,b]_{\rho}$. We apply the Jacquet functor ${}_{N_j}$ on the second factor, and so, by
\[  \mathrm{St}(\Delta')_{N_j}=\mathrm{St}([\nu_{\rho}\cdot b(\Delta), b(\Delta')]) \boxtimes \mathrm{St}(\Delta'') ,
\]
we have
\[  D_{\Delta'}(\pi) \boxtimes  \mathrm{St}([\nu_{\rho}\cdot b(\Delta), b(\Delta')]) \boxtimes \mathrm{St}(\Delta'') \hookrightarrow \pi_{N_{n-i,i-j,j}} .
\]
By taking Jacquet functor in stages and applying Frobenius reciprocity, we have that $D_{\Delta''}(\pi)\neq 0$. This gives a contradiction. 
\end{proof}



Recall that  $\mathrm{Mult}_{\Delta-sat}$ is defined in Section \ref{ss eta invariant}.

\begin{lemma} \label{lem big derivatives on equal case}
Let $\pi \in \mathrm{Alg}_f$. Let $\Delta \in \mathrm{Seg}$ and let $\mathfrak p \in \mathrm{Mult}_{\Delta-sat}$. 
Let $\Delta'$ be a segment such that $a(\Delta) \leq a(\Delta')$ and $b(\Delta) \leq b(\Delta')$. Suppose $D_{\Delta'}(I_{\mathfrak p}(\pi))\neq 0$. Suppose, for any segment $\widetilde{\Delta}$ with the following two properties:
\begin{itemize}
  \item $b(\widetilde{\Delta})\cong b(\Delta)$; and
  \item $a(\widetilde{\Delta}) \leq b(\Delta)$,
\end{itemize}
we have $\mathbb D_{\widetilde{\Delta}}(\pi) =0$. Let $\Delta''=[a(\Delta'), b(\Delta)]$. Then
\[  \mathbb D_{\Delta'}(\pi \times \mathrm{St}(\mathfrak p)) \cong \mathbb D_{[\nu\cdot b(\Delta), b(\Delta')]}(\pi) \times \mathrm{St}(\mathfrak p-\Delta'').
\]
\end{lemma}

\begin{proof}
Let $i=l_{abs}(\Delta')$. Recall that 
\[(*)\quad  \mathbb D_{\Delta'}(\pi \times \mathrm{St}(\Delta')) \cong \mathrm{Hom}_{G_i}(\mathrm{St}(\Delta') , (\pi \times \mathrm{St}(\mathfrak p))_{N_i}) .\]
Let $m=l_{abs}(\mathfrak p)$. Then the layers in the geometric lemma of $(\pi \times \mathrm{St}(\mathfrak p))_{N_i}$ take the form: for $k+l=i$,
\[  \omega_{k,l}=\mathrm{Ind}_{P_{n-k, m-l} \times P_{k, l}}^{G_{n+m-i}\times G_i}((\pi)_{N_k} \boxtimes (\mathrm{St}(\mathfrak p))_{N_l})^{\phi} ,
\]
where $\phi$ is a twist takes a $G_{n-k}\times G_k \times G_{m-l} \times G_l$-representation to a $G_{n-k}\times G_{m-l} \times G_k \times G_l$-representation. 

Write $\Delta'=[a',b']_{\rho'}$. By Frobenius reciprocity, we have that:
\[   \mathrm{Hom}_{ G_i} ( \mathrm{St}(\Delta'), \omega_{k,l}) \cong \mathrm{Hom}_{G_k \times G_l}(\tau_k \boxtimes \widetilde{\tau}_l,  \widetilde{\omega}_{k,l}   ) 
\]
where
\[   \tau_k=\mathrm{St}([\nu_{\rho'}^{-(k-1)}b(\Delta'), b(\Delta')]), \quad \widetilde{\tau}_l=\mathrm{St}([a(\Delta'), \nu_{\rho'}^{l-1}a(\Delta')]) 
\]
and
\[  \widetilde{\omega}_{k,l}=\mathrm{Ind}_{P_{n-k, m-l}\times P_{k,l}}^{G_{n+m-i}\times G_i} ((\pi)_{N_k} \boxtimes (\mathrm{St}(\mathfrak p))_{N_l})^{\phi} .
\]
Thus we have that:
\[(**)\quad   \mathrm{Hom}_{ G_i} ( \mathrm{St}(\Delta'), \omega_{k,l}) \cong\mathbb D_{\tau_k}(\pi) \times \mathbb D_{\widetilde{\tau}_l}(\mathrm{St}(\mathfrak p)) .
\]

 Let $l^*=l_{abs}([a(\Delta'), b(\Delta)])$ and $k^*=l_{abs}(\Delta')-l^*$.  Note that if $l\neq l^*$, then either 
\[  \mathbb{D}_{\tau_l}(\pi) =0 \quad \mbox{ or } \quad \mathbb D_{\widetilde{\tau}_k}(\mathrm{St}(\mathfrak p))= 0 ,
\]
which follows by either the assumption in the lemma or a comparison of the cuspidal support.



Now combining (*), (**) and the claim, we have that
\[  \mathbb D_{\Delta'}(\pi \times \mathrm{St}(\Delta')) \cong \mathbb D_{\tau_{k^*}}(\pi)\times \mathbb D_{\widetilde{\tau}_{l^*}}(\mathrm{St}(\mathfrak p))
\]
Now the lemma follows from Lemma \ref{lem big derivative st}.
\end{proof}

\begin{lemma} \label{lem unlinked commute derivative}
Let $\pi \in \mathrm{Alg}_f$. Let $\mathfrak p \in \mathrm{Mult}_{\Delta-sat}$. Let $\Delta'$ be a segment satisfying one of the following conditions:
\begin{enumerate}
\item $a(\Delta')<a(\Delta)$ and $b(\Delta)\leq b(\Delta')$; or
\item $b(\Delta')<a(\Delta)$; or
\item $b(\Delta)<a(\Delta')$.
\end{enumerate}
 Then
\[  \mathbb{D}_{\Delta'}(\pi \times \mathrm{St}(\mathfrak p))=\mathbb{D}_{\Delta'}(\pi)\times \mathrm{St}(\mathfrak p).
\]
\end{lemma}

\begin{proof}
This again follows by using the geometric lemma and to notice the only layer for contributing the big derivative. We omit the details.
\end{proof}

We now generalize Proposition \ref{prop multiplicity one segment case} to arbitrary generic representations. In order to use induction, one uses Section \ref{ss eta invariant} and some properties in Proposition \ref{prop composition factor mx multiseg}. For two segments $\Delta, \Delta'$, we write $\Delta \leq_a \Delta'$ if either one of the conditions hold:
\begin{itemize}
\item $a(\Delta) < a(\Delta')$; or
\item $a(\Delta)\cong a(\Delta')$ and $b(\Delta)\leq b(\Delta')$.
\end{itemize}

\begin{theorem} \label{thm generic socle irreducible}
Let $\mathfrak m \in \mathrm{Mult}$ be generic. Let $\sigma=\mathrm{St}(\mathfrak m)$. Then, for any $\pi \in \mathrm{Irr}$ with $D_{\sigma}(\pi) \neq 0$, $\mathbb D_{\sigma}(\pi)$ is SI. 
\end{theorem}

\begin{proof}

We shall prove by an induction on the number of segments in $\mathfrak m$. When there is only one segment in $\mathfrak m$, it follows from Proposition \ref{prop multiplicity one segment case}.

We consider the set
\[ \mathcal B=\left\{ b(\Delta) : \Delta \in \mathfrak m  \right\}.
\]
Then we choose a minimal element $\rho$ in $\mathcal B$ with respect to $\leq$ (see the ordering in Section \ref{ss basic notations}). Among those strongly $\rho$-saturated segments, we choose a $\leq_a$-maximal segment $\Delta^*$ (equivalently the shortest one among those).

Let 
\[ \mathfrak n=\left\{ \widetilde{\Delta} : a(\widetilde{\Delta})\not\cong a(\Delta^*) \right\} \cup \left\{ \widetilde{\Delta} \setminus \Delta^* : a(\widetilde{\Delta}) \cong a(\Delta^*), \widetilde{\Delta} \in \mathfrak m  \right\} \]
Here $\widetilde{\Delta}\setminus \Delta^*$ is the set-theoretic subtraction i.e. for writing $\Delta^*=[a^*,b^*]_{\rho'}$ and $\widetilde{\Delta}=[a^*, \widetilde{b}]_{\rho'}$, 
\[ \widetilde{\Delta}\setminus \Delta^* = [b^*+1, \widetilde{b}]_{\rho'}. 
\]

\noindent
{\it Claim 1}: $\mathfrak n$ is generic. \\
{\it Proof of Claim 1}: It follows form the choice of $\Delta^*$ that there is no segment 
$\widetilde{\Delta}$ in $\mathfrak m$ such that $\widetilde{\Delta} \subsetneq \Delta^*$. Then it is direct to check from the genericity of $\mathfrak m$ that $\mathfrak n$ is also generic. \\

We shall use Claim 1 later. We now consider some other multisegments:
\[  \mathfrak o = \left\{ \widetilde{\Delta} \in \mathfrak m: a(\widetilde{\Delta}) \cong a(\Delta)     \right\} ,\quad  \mathfrak o' = \left\{ \widetilde{\Delta}\setminus \Delta : \widetilde{\Delta} \in \mathfrak o \right\}, \]
\[  \mathfrak p= \mathfrak m-\mathfrak o . \]

Let $\mathfrak t=\mathfrak{mx}(\mathfrak m, \Delta^*)$ and let $\tau=D_{\mathfrak t}(\pi)$. let $k$ be the number of segments in $\mathfrak t$. We now prove another claim:

\noindent
{\it Claim 2:} $\mathbb D_{\mathfrak o}(\tau \times \mathrm{St}(\mathfrak t)) \cong \mathbb D_{\mathfrak o'}(\tau) \times \mathrm{St}(\mathfrak t-k\cdot \Delta)$. \\
\noindent
{\it Proof of Claim 2:} We shall prove inductively on the number of segments. When there is only one segment in $\mathfrak o$, it follows from Lemmas \ref{lem derivative zero for mx zero} and \ref{lem big derivatives on equal case}. We suppose there are more than one segments. In such case, we pick a longest segment $\overline{\Delta}$ in $\mathfrak o$ and hence $b(\widetilde{\Delta}) \leq b(\overline{\Delta})$ for any $\widetilde{\Delta} \in \mathfrak o$. This also implies 
\begin{align} \label{eqn subset of segments}
     \widetilde{\Delta}\setminus \Delta \subset  \overline{\Delta}
\end{align}
for any $\widetilde{\Delta} \in \mathfrak o$. 

Now induction gives that 
\[ \mathbb D_{\mathfrak o-\overline{\Delta}}(\tau) \times \mathrm{St}(\mathfrak t)) \cong \mathbb D_{\mathfrak o''}(\tau) \times \mathrm{St}(\mathfrak t-(k-1)\cdot \Delta) ,
\]
where $\mathfrak o=\left\{ \widetilde{\Delta}\setminus \Delta \right\}_{\widetilde{\Delta} \in \mathfrak o-\overline{\Delta}}$. Now the claim will follow from Lemma \ref{lem big derivatives on equal case} if we can verify that
\[    \quad \mathbb D_{\overline{\Delta}}(\mathbb D_{\mathfrak o-\overline{\Delta}}(\tau))=0 . 
\]
To this end, we use (\ref{eqn subset of segments}) and so we can apply Lemma \ref{lem commutativity unlinked seg} (several times) to obtain 
\[   \mathbb D_{\overline{\Delta}}(\mathbb D_{\mathfrak o-\overline{\Delta}}(\tau)=\mathbb D_{\mathfrak o-\overline{\Delta}}\circ \mathbb D_{\overline{\Delta}}(\tau)=0 .
\]

Now now prove another claim:

\noindent
{\it Claim 3:} $\mathbb D_{\mathfrak p}\circ \mathbb D_{\mathfrak o}(\tau \times \mathrm{St}(\mathfrak t)) \cong \mathbb D_{\mathfrak o'+\mathfrak p}(\tau) \times \mathrm{St}(\mathfrak t-k\cdot \Delta)$. \\
\noindent
{\it Proof:} It follows from Claim 2 that we only have to prove:
\[  \mathbb D_{\mathfrak p}(\mathbb D_{\mathfrak o'}(\tau) \times \mathrm{St}(\mathfrak t-k\cdot \Delta)) \cong \mathbb D_{\mathfrak o'+\mathfrak p}(\tau) \times \mathrm{St}(\mathfrak t-k\cdot \Delta) . \]
This follows by using Lemma \ref{lem unlinked commute derivative} several times. \\

Now to prove $D_{\sigma}(\pi)$ appears with multiplicity one in the Jordan-H\"older series of $\mathbb D_{\sigma}(\pi)$, it suffices to show that $D_{\sigma}(\pi)$ appears with multiplicity one in $\mathbb D_{\sigma}(D_{\mathfrak p}(\pi)\times \mathrm{St}(\mathfrak p))$. Now Claim 2 reduces to prove that (**) $\mathbb D_{\mathfrak o'+\mathfrak p}(\tau) \times \mathrm{St}(\mathfrak t-k\cdot \Delta)$.

Before proving (**), we have one more claim:

\noindent
{\it Claim 4:} $\mathbb D_{\mathfrak o'}(\tau)$ and $\mathbb D_{\mathfrak o'+\mathfrak p}(\tau)$ are $\Delta$-reduced. \\
{\it Proof of Claim 4:} With a similar argument to proving Claim 2, we have:
\[  \mathbb D_{\mathfrak t-k\dot \Delta+\mathfrak o}(\tau \times \mathrm{St}(\mathfrak t)) \cong \mathbb   D_{\mathfrak o'}(\tau) 
\]
We see that the LHS is $\Delta$-reduced (since $\mathbb D_{\mathfrak t-k\dot \Delta+\Delta'+\mathfrak o}(\tau \times \mathrm{St}(\mathfrak t))=0$ for a $\Delta$-saturated segment $\Delta'$) and so is the RHS. Now it follows from Lemma \ref{lem commutativity unlinked seg} that we also have $\mathfrak{mx}(\mathbb D_{\mathfrak p}\circ \mathbb D_{\mathfrak o'}(\tau), \Delta)=\emptyset$.  \\

Let $\mathfrak t'=\mathfrak t-k\cdot \Delta$. \\

Claim 1 shows that $\mathbb D_{\mathfrak o'+\mathfrak p}(\tau)$ is SI by induction. With Claim 4, the SI property of $\mathbb D_{\mathfrak o'+\mathfrak p}(\tau) \times  \mathrm{St}(\mathfrak t')$  now follows from \cite[Lemma 7.1]{LM22} and so we also have the SI property for $ \mathbb D_{\mathfrak o+\mathfrak p}(\tau \times \mathrm{St}(\mathfrak t))$ by Claim 3. Since $\pi \hookrightarrow \tau \times \mathrm{St}(\mathfrak t)$, we now also have the SI property of $\mathbb D_{\sigma}(\pi)$. 
\end{proof}

\begin{remark}
We remark that \cite[Corollary 3.7]{KKKO15} shows that there is a unique simple submodule of $\mathbb D_{\sigma}(\pi)$ for $\sigma \in \mathrm{Irr}^{\square}$ and $\pi \in \mathrm{Irr}$. Indeed, for the special case of generic representations, it also follows from  \cite[Proposition 2.5]{Ch21} (also see \cite{Ch22+}), using some inputs from branching laws.
\end{remark}


\begin{thebibliography}{AGRS}
\bibitem[AG12]{AG12} Aizenbud, A., Gourevitch, D. (2012). Multiplicity Free Jacquet Modules. Canadian Mathematical Bulletin, 55(4), 673-688. doi:10.4153/CMB-2011-127-8\bibitem[Au95]{Au95} Aubert, A.-M.: Dualit\'e dans le groupe de Grothendieck de la cat\'egorie des repr\'esentations lisses delongueur finie d'un groupe r\'eductif p-adique. Trans. Am. Math. Soc. 347(6), 2179-2189 (1995) https://doi.org/10.1090/S0002-9947-1995-1285969-0 
\bibitem[AL22]{AL22}  A. Aizenbud, E. Lapid, A binary operation on irreducible components of Lusztig's nilpotent varieties I: definition and properties, to appear in Pure and Applied Mathematics Quarterly.
\bibitem[Ba02]{Ba02} A. I. Badulescu, Correspondance de Jacquet-Langlands pour les corps locaux de ca
ract\'eristique non nulle, Ann. Sci. \'Ecole Norm. Sup. (4) 35 (2002), no. 5, p. 695-747.
\bibitem[BLM13]{BLM13} Badulescu, I., Lapid, E., M\'inguez, A.: Une condition suffisante pour l'irr\'eductibilit\'e d'une induite parabolique de GL(m, D). Ann. Inst. Fourier (Grenoble) 63(6), 2239-2266 (2013)
\bibitem[Be92]{Be92} J. Bernstein. Represenations of p-adic groups. Harvard University, 1992. Lectures by Joseph Bernstein. Written by Karl E. Rumelhart.
\bibitem[BBK18]{BBK18} Bernstein, J., Bezrukavnikov, R. and Kazhdan, D. Deligne-Lusztig duality and wonderful compactification. Sel. Math. New Ser. 24, 7-20 (2018). https://doi.org/10.1007/s00029-018-0391-5
\bibitem[BHLS10]{BHL10} Badulescu, A. I., G. Henniart, B. Lemaire, and V. S\'echerre. Sur Le Dual Unitaire de GLr(D). American Journal of Mathematics 132, no. 5 (2010): 1365-96. http://www.jstor.org/stable/40864614.
\bibitem[BZ76]{BZ76} Bernstein, I.N., Zelevinski, A.V., Representations of the group GL(n, F), where F is a local non-Archimedean field. Uspehi Mat. Nauk (3) 31(189), 5-70 (1976)
\bibitem[BZ77]{BZ77} I. N. Bernstein and A. V. Zelevinsky, {\it Induced representations of reductive p-adic groups}, I, Ann. Sci. Ecole Norm. Sup. {\bf 10} (1977), 441-472.
\bibitem[Ch16]{Ch16} K.Y. Chan, Duality for Ext-groups and extensions of discrete series for graded Hecke algebras, Adv. Math. 294 (2016), 410-453.
\bibitem[Ch18]{Ch18}  Chan, K.Y. Some Methods of Computing First Extensions Between Modules of Graded Hecke Algebras. Algebr Represent Theor 21, 859-895 (2018). https://doi.org/10.1007/s10468-017-9742-8
\bibitem[Ch21]{Ch21} K.Y. Chan, Homological branching law for $(\mathrm{GL}_{n+1}(F),\mathrm{GL}_n(F))$: projectivity and indecomposability, Invent. math. (2021). https://doi.org/10.1007/s00222-021-01033-5
\bibitem[Ch23]{Ch21+} K.Y. Chan, Ext-multiplicity theorem for standard representations of $(\mathrm{GL}_{n+1}(F),  \mathrm{GL}_\mathrm{deg}()F))$, preprint 2021, arxiv.org/abs/2104.11528.
\bibitem[Ch22]{Ch22} K.Y. Chan, Restriction for general linear groups: The local non-tempered Gan-Gross-Prasad conjecture (non-Archimedean case), Crelles Journal, vol. 2022, no. 783, 2022, pp. 49-94. https://doi.org/10.1515/crelle-2021-0066
\bibitem[Ch22+]{Ch22+} K.Y. Chan, Construction of simple quotients of Bernstein-Zelevinsky derivatives and highest derivative multisegments I: reduction to combinatorics, preprint (2022)
\bibitem[Ch22+b]{Ch22+b} K.Y. Chan, On commutations of derivatives and integrals of $\square$-irreducible representations for $p$-adic GL, arXiv:2210.17249 (2022)
\bibitem[Ch22+c]{Ch22+c} K.Y. Chan, Quotient branching law for $p$-adic $(\mathrm{GL}_{n+1}, \mathrm{GL}_n)$ I: generalized Gan-Gross-Prasad relevant pairs, arXiv:2212.05919 (2022)
\bibitem[CS21]{CS21} K.Y. Chan and G. Savin, A vanishing Ext-branching theorem for $(\mathrm{GL}_{n+1}(F), \mathrm{GL}_n(F))$, Duke Math Journal, 2021, 170 (10), 2237-2261. https://doi.org/10.1215/00127094-2021-0028
\bibitem[DKV84]{BKV84} Deligne, P. , Kazhdan, D., Vign\'eras, M.-F.: Repr\'esentations des Alg\`ebres Centrales Simples p-Adiques, Representations of Reductive Groups Over a Local Field, pp. 33-117 (1984)
\bibitem[GGP20]{GGP20} W.T. Gan, B.H. Gross, and D. Prasad, Branching laws for classical groups: the non-tempered case, Compositio Mathematica, 156(11) (2020), 2298-2367. doi:10.1112/S0010437X20007496
\bibitem[GK75]{GK} Gel'fand, I.M., Kajdan, D.A.: Representations of the group GL(n, K) where K is a local field, Lie groups and their representations (Proc. Summer School, Bolyai J\'anos Math. Soc., Budapest, 1971), pp. 95-118 (1975)
\bibitem[GLS11]{GLS11} Christof Gei\ss, Bernard Leclerc, Jan Schr\"oer, Kac-Moody groups and cluster algebras, Advances in Mathematics, Volume 228, Issue 1, 2011, Pages 329-433, doi.org/10.1016/j.aim.2011.05.011.
\bibitem[GV01]{GV01} Grojnowski, I., Vazirani, M. Strong multiplicity one theorems for affine Hecke algebras of type A. Transformation Groups 6, 143-155 (2001). https://doi.org/10.1007/BF01597133
\bibitem[GL21]{GL21} M. Gurevich and E. Lapid, Robinson-Schensted-Knuth correspondence in the representation theory of the general linear group over a non-archimedean local field, Represent. Theory 25 (2021), 644-678, with an appendix by Mark Shimozono.
\bibitem[HM08]{HM08} Marcela Hanzer, Goran Mui\'c, On an algebraic approach to the Zelevinsky classification for classical p-adic groups, Journal of Algebra, Volume 320, Issue 8, 2008, Pages 3206-3231, ISSN 0021-8693, https://doi.org/10.1016/j.jalgebra.2008.07.002.
\bibitem[HL13]{HL13} Hernandez, D., Leclerc, B. (2013). Monoidal Categorifications of Cluster Algebras of Type A and D . In: Iohara, K., Morier-Genoud, S., R\'emy, B. (eds) Symmetries, Integrable Systems and Representations. Springer Proceedings in Mathematics \& Statistics, vol 40. Springer, London. doi.org/10.1007/978-1-4471-4863-08
\bibitem[Ja07]{Ja07} Jantzen, C.: Jacquet modules of p-adic general linear groups. Represent. Theory 11, 45-83 (2007). (electronic)
\bibitem[Ka17]{Ka17} Kato, Syu. An algebraic study of extension algebras. American Journal of Mathematics 139, no. 3 (2017): 567-615. doi:10.1353/ajm.2017.0015.
\bibitem[KKKO15]{KKKO15} Seok-Jin Kang, Masaki Kashiwara, Myungho Kim and Se-jin Oh, Simplicity of heads and socles of tensor products, Compos. Math. 151 (2015), no. 2, 377-396.
\bibitem[KKKO18]{KKKO18}  Seok-Jin Kang, Masaki Kashiwara, Myungho Kim and Se-jin Oh, J. Amer. Math. Soc. 31 (2018), 349-426, doi.org/10.1090/jams/895
\bibitem[Kl15]{Kl15} Alexander S. Kleshchev, Affine highest weight categories and affine quasihereditary algebras, Proc. Lond. Math. Soc. (3) 110 (2015), no. 4, 841-882, DOI 10.1112/plms/pdv004
\bibitem[LM16]{LM16} E. Lapid, A. M\'inguez, On parabolic induction on inner forms of the general linear group over a non-Archimedean local field, Sel. Math. New Ser. (2016) 22, 2347-2400.
\bibitem[LM19]{LM19}  E. Lapid, A. M\'inguez, Geometric conditions for $\square$-irreducibility of certain representations of the general linear group over a non-archimedean local field, Advances in Mathematics Volume 339 (2018), p. 113-190, 10.1016/j.aim.2018.09.027
\bibitem[LM20]{LM20} Lapid, E., M\'inguez, A. Conjectures and results about parabolic induction of representations of $\mathrm{GL}_\mathrm{deg}()F)$. Invent. math. 222, 695-747 (2020). https://doi.org/10.1007/s00222-020-00982-7
\bibitem[LM22]{LM22} E. Lapid, M\'inguez,, A binary operation on irreducible components of Lusztig's nilpotent varieties II: applications and conjectures for representations of GLn over a non-archimedean local field, to appear in Pure and Applied Mathematics Quarterly.
\bibitem[Le03]{Le03} B. Leclerc, Imaginary vectors in the dual canonical basis of Uq(n), Transform. Groups 8 (2003), no. 1, 95-104. MR 1959765
\bibitem[Ma95]{Ma63} S. MaeLane, Homology; Springer Verlag 1995 https://doi.org/10.1007/978-3-642-62029-4
\bibitem[Me06]{Me06} R. Meyer, Homological algebra for Schwartz algebras of reductive p-adic groups
Noncommutative Geometry and Number Theory, Aspects of Mathematics E, vol. 37, Vieweg Verlag, Wiesbaden (2006), pp. 263-300
\bibitem[Mi09]{Mi09} M\'inguez, A.: Sur irr\'eductibilit\'e d'une induite parabolique. J. Reine Angew. Math. 629, 107-131 (2009)
\bibitem[MS13]{MS13} M\'inguez, A., S\'echerre, V. (2013). Repr\'esentations banales de ${\mathrm GL}_{m}({\mathrm D})$. Compositio Mathematica, 149(4), 679-704. doi:10.1112/S0010437X12000590
\bibitem[MS14]{MS14} Alberto M\'inguez, Vincent S\'echerre, Repr\'esentations lisses modulo $l$ de GLm(D), Duke Mathematical Journal, Duke Math. J. 163(4), 795-887,  DOI: 10.1215/00127094-2430025
\bibitem[MW86]{MW86} M\oe glin, C., Waldspurger, J.-L.: Sur l'involution de Zelevinski. J. Reine Angew. Math. 372, 136-177 (1986)
\bibitem[OS09]{OS09} Eric Opdam, Maarten Solleveld, Homological algebra for affine Hecke algebras,
Advances in Mathematics, Volume 220, Issue 5, 2009, Pages 1549-1601, https://doi.org/10.1016/j.aim.2008.11.002.
\bibitem[OS13]{OS13} Opdam, E., Solleveld, M. Extensions of tempered representations. Geom. Funct. Anal. 23, 664-714 (2013). https://doi.org/10.1007/s00039-013-0219-6
\bibitem[Pr18]{Pr18} D. Prasad, An Ext-analogue of branching laws, ICM proceedings 2018.
\bibitem[Sc09]{Sc09} S\'echerre, Vincent. Proof of the Tadi\'c conjecture (U0) on the unitary dual of GLm(D), vol. 2009, no. 626, 2009, pp. 187-203. https://doi.org/10.1515/CRELLE.2009.007
\bibitem[SS97]{SS97} P. Schneider, U. Stuhler, Representation theory and sheaves on the Bruhat-Tits building, Publ. Math. Inst. Hautes \'Etudes Sci. 85 (1997) 97-191.
\bibitem[Si79]{Si79} A. J. Silberger, Introduction to Harmonic Analysis on Reductive p-adic Groups, Math. Notes. 23, Princeton University Press, Princeton, 1979.
\bibitem[Sta]{St} The Stacks Project Authors, Stacks Project, https://stacks.math.columbia.edu, 2020
\bibitem[Ta90]{Ta90} Tadic, Marko. Induced representations of GL (n, A) for p-adic division algebras A. Journal f\"ur die reine und angewandte Mathematik (Crelles Journal), vol. 1990, no. 405, 1990, pp. 48-77. https://doi.org/10.1515/crll.1990.405.48
\bibitem[Ta95]{Ta95} M. Tadic, Structure Arising from Induction and Jacquet Modules of Representations of Classical p-Adic Groups, Journal of Algebra, Volume 177, Issue 1, 1995, Pages 1-33, https://doi.org/10.1006/jabr.1995.1284.
\bibitem[Ze80]{Ze80} A. Zelevinsky, Induced representations of reductive p-adic groups II, Ann. Sci. Ecole Norm. Sup. {\bf 13} (1980), 154-210.
\bibitem[Ze81]{Ze81} A. Zelevinsky, Representations of finite classical groups, a Hopf algebra approach, Lecture Notes in Math., Vol. 869, Springer-Verlag, Berlin, 1981.
\end{thebibliography}
\end{document}